\documentclass[12pt,a4paper]{article}%
\usepackage{amsmath}
\usepackage{graphicx}
\usepackage{epsfig}%
\usepackage{amsfonts}%
\usepackage{amssymb}
\usepackage{algorithm}
\usepackage{algorithmic}
\usepackage{color}
\usepackage{tikz} \usetikzlibrary{3d,positioning,shapes}
\usetikzlibrary{scopes,shapes.geometric,shadows}
\usepackage{pgfplots} \usetikzlibrary{plotmarks}
\usetikzlibrary{calc,backgrounds,fit}

 \usepackage{tikz} \usetikzlibrary{3d,positioning,shapes}
 \usetikzlibrary{scopes,shapes.geometric,shadows}
 \usepackage{pgfplots} \usetikzlibrary{plotmarks}

\newtheorem{theorem}{Theorem}[section]

\newtheorem{lemma}[theorem]{Lemma}

\newtheorem{remark}[theorem]{Remark}

\newenvironment{proof}[1][Proof]{\noindent \emph{#1.} }{\hfill \
\rule{0.5em}{0.5em}}

\makeatletter\@addtoreset{equation}{section}\makeatother
\makeatletter\@addtoreset{figure}{section}\makeatother
\makeatletter\@addtoreset{table}{section}\makeatother
\begin{document}

\title{Fast iterative solution of the Bethe-Salpeter eigenvalue problem 
using low-rank and QTT tensor approximation}

\author{Peter Benner\thanks{Max Planck Institute for Dynamics of Complex Technical Systems, Sandtorstr.~1, 
D-39106 Magdeburg, Germany
({\tt benner@mpi-magdeburg.mpg.de})}\and
Sergey Dolgov\thanks{University of Bath, The Avenue, Bath, BA2 7AY, United Kingdom ({\tt S.Dolgov@bath.ac.uk}). 
S. Dolgov gratefully acknowledges funding
from the Engineering and Physical Sciences Research Council (EPSRC) Fellowship EP/M019004/1. 
This work was mainly conducted when S. Dolgov worked at the Max Planck Institute for 
Dynamics of Complex Technical Systems, Magdeburg, Germany.} \and
   Venera Khoromskaia\thanks{Max Planck Institute for
        Mathematics in the Sciences, Leipzig;
        Max Planck Institute for Dynamics of Complex Systems, Magdeburg, Germany ({\tt vekh@mis.mpg.de}).}
        \and Boris N. Khoromskij\thanks{Max Planck Institute for
        Mathematics in the Sciences, Inselstr.~22-26, D-04103 Leipzig,
        Germany ({\tt bokh@mis.mpg.de}).}\\        
        }

%\date{}

\maketitle
\begin{abstract}
In this paper, we  study and implement the structural iterative eigensolvers for the large-scale
eigenvalue problem in the Bethe-Salpeter equation (BSE) based on the reduced basis approach
via low-rank factorizations in generating matrices, introduced in the previous paper.
% The solution of TDA equation is chosen as the initial guess. 
The approach reduces numerical costs down to $\mathcal{O}(N_b^2)$ 
in the size of atomic orbitals basis set, $N_b$, instead of 
practically intractable $\mathcal{O}(N_b^6)$ complexity scaling for
the direct diagonalization of the BSE matrix.
As an alternative to rank approximation of the static screen interaction 
part of the BSE matrix, we propose to restrict it to a small active sub-block, 
with a size balancing the storage for rank-structured representations
of other matrix blocks. We demonstrate that the enhanced reduced-block approximation  
exhibits higher precision within the controlled numerical cost,
providing as well a distinct two-sided error estimate for the BSE eigenvalues. 
It is shown that further reduction of the asymptotic computational cost 
is possible due to ALS-type iteration in block tensor train (TT) format applied to
the quantized-TT (QTT) tensor representation of both long eigenvectors 
and rank-structured matrix blocks.
The QTT-rank of these entities possesses  almost the same magnitude as 
the number of occupied orbitals in the molecular systems, $N_o$, hence
the overall asymptotic complexity for solving the BSE problem can be 
estimated by $\mathcal{O}(\log(N_o) N_o^{2})$.
We confirm numerically a considerable decrease in computational time 
for the presented iterative approach applied to various compact and
chain-type molecules, while supporting sufficient accuracy.  

\end{abstract}



\section{Introduction}\label{Introd:MP2}

%\textcolor{red}{What is a bib entry ``DeSaStJaCoLo:12''?} BSE_Fock2.bib

This paper continues the previous article \cite{BeKhKh_BSE:15} where the reduced basis approach
to the solution of Bethe-Salpeter algebraic eigenvalue problem 
was introduced, based on the idea of 
low-rank plus diagonal approximation to the matrix blocks and then solving the 
small size spectral problem via Galerkin projection onto the reduced basis set.

The Bethe-Salpeter equation (BSE) \cite{BeSa:51}, \cite{Hedin}  offers one of the 
commonly used mathematical models for {\it ab initio} computation of the 
absorption spectra for molecules or surfaces of solids, see also
\cite{ReOlRuOni:02,OniReRu:02,SchGluHaBe:03,SBotti:14,ReToTeHeSa:15}.
The BSE approach leads to the challenging computational task on the solution of a large eigenvalue
problem for fully populated (dense) matrix, that, in general, is non-symmetric.
 The size of the BSE matrix scales  quadratically $\mathcal{O}(N^2_b)$
in a size $N_b$ of the atomic orbitals basis sets,  commonly used in \emph{ab initio} electronic 
structure calculations. 
Hence, the direct diagonalization of $\mathcal{O}(N_b^{6})$-complexity becomes 
prohibited even for moderate size molecules. 

Methods for solving partial eigenvalue problems for matrices with a special structure 
as in the BSE eigenvalue problem 
have been intensively studied in the literature.  These structures are related to the so 
called Hamiltonian matrices, exposing a particular block pattern.  
Papers and books treating Hamiltonian 
eigenvalue problems include \cite{BeFa:97,BeMeXu:98,Kressner_book:05,FaKre:06}, 
see also the recent survey \cite{BGFa:15} and the references therein. 
Special cases of the BSE and other eigenvalue problems related to Hartree-Fock approximations 
lead to anti-block-diagonal 
Hamiltonian eigenproblems that can be solved by special techniques based on minimization 
principles \cite{BaiLi:12,BaiLi:13}. The algebraic structure of the BSE matrix is not that 
of a Hamiltonian matrix in the general case, but yields a so called complex $J$-symmetric 
matrix.  Theory and numerical solution of such eigenvalue problems are discussed in
\cite{BunByeMehrm:92,MMT03,MMMM,Meh08,BeFaYa:15}, where the particular instance of the BSE 
matrix is considered in \cite{BeFaYa:15}.  Other structural eigensolvers tailored for electronic 
structure calculations are discussed 
in \cite{RoGeSaBa:08,RoLuGa:10,DeSaStJaCoLo:12,NaPoSaad:13,LinSaadYa:15,ShJoYaDeLo:16}.

% The use of diagonal plus low-rank matrix representations
% provides mean for efficient matrix-vector algebra in the framework of structured iterative
% solvers based on Krylov subspace type algorithms.
Recall that \cite{BeKhKh_BSE:15} introduces and studies
a reduced basis method for the approximate numerical solution of
the BSE algebraic eigenvalue problem that is well suited for Krylov subspace type algorithms. 
This approach is based on model reduction via projection 
onto a {\it reduced basis}, which is constructed by using the eigenvectors
of a simplified  system matrix obeying a diagonal plus low-rank data-sparse structure.
The  reduced basis method in \cite{BeKhKh_BSE:15} includes two main computational steps. 
First, the diagonal plus low-rank approximation to the fully populated
blocks in the BSE matrix is calculated, enabling an easier partial eigenvalue solver for a large
auxiliary system relying only on matrix-vector multiplications with rank-structured matrices.
Second, a small subset of eigenvectors from the auxiliary eigenvalue problem
is selected to build the Galerkin projection of the exact BSE system onto this
reduced basis set. The adaptive choice of the rank parameters is determined by the 
$\varepsilon$-thresholding in the matrix factorizations.

Following \cite{BeKhKh_BSE:15}, we use the particular description of the BSE matrix presented
in \cite{ReTouSa1:13}. %,ReTouSa:13},
We build up the BSE system matrix by using the complete output of
the Hartree-Fock calculations including rank-structured representation
of the two electron integrals (TEI) tensor in the molecular orbital basis 
precomputed by a grid-based tensor approach \cite{vekh:13,VeKhBoKhSchn:12,VeKhorMP2:13,VeKhorTromsoe:15}.

In this paper, we study and implement the structured iterative solvers for the large-scale
BSE eigenvalue problem, based on reduced basis approach
via low-rank factorizations in generating matrices \cite{BeKhKh_BSE:15}.
As the alternative  to problematic low-rank approximation of the static screen interaction 
part in the BSE matrix we propose to implement the matrix-vector product 
with this matrix block  using  only its  restriction to a small sub-block, 
with a size that balances the complexity of rank-structured representations
of the other parts in the system matrix. We show numerically that this enhanced representation considerably 
improves accuracy of the solution under the controlled numerical cost.
The approach reduces the numerical expense of the direct diagonalization down to $\mathcal{O}(N_b^2)$ 
in the size of the atomic orbitals basis set, $N_b$. 

Several iterative schemes are considered for both the 
Tamm-Dancoff approximation (TDA) and the full BSE
$2\times 2$ block system. The most efficient subspace iteration is based on
application of the matrix inverse, which for our matrix formats can be evaluated in a structural form by 
using the Sherman-Morrison formula. Numerical tests confirm the considerable 
decrease in computational time for the presented approach, 
while supporting the sufficient accuracy.

Further reduction of the numerical cost can be achieved by adapting the
ALS-type iteration (in particular, DMRG iteration)
for computing the eigenvectors in the block-QTT 
tensor representation \cite{DoKhSavOs_mEIG:13}, where the skeleton vectors of a 
low-rank part of the matrix are further approximated in the QTT format.
Application of the QTT-approximation is motivated by the observation, known
from  \cite{VeKhorMP2:13}, that the  generating Cholesky factors 
in the TEI tensor exhibit the average QTT-ranks proportional 
only to the number of occupied orbitals in the molecular system, $N_{o}$,
but they do not depend on the total BSE matrix size, $\mathcal{O}(N_b^2)$.
For eigenvectors in the block-QTT format, the QTT ranks are even smaller, 
typically proportional to the number of the sought eigenvectors,
which makes this approach to the BSE very competitive.

 The rest of the paper is organized as follows. In Section \ref{sec:RBA_BSE} we recall the
reduced basis approach to BSE problem introduced in  \cite{BeKhKh_BSE:15},
%iterative solution of the auxiliary eigenvalue problem approximation the original BSE equation.
%First, we recall the  reduced basis approach to BSE problem introduced in  \cite{BeKhKh_BSE:15},
based on {low-rank} factorization of the BSE matrix blocks. 
Next, in Section \ref{sec:Iter_RedBlock_BSE} we describe the enhanced structural 
representation of the BSE system matrix by the reduced-block approximation to the 
the static screen interaction sub-matrix. This gainfully supplements the diagonal
plus low-rank decompositions of the remaining part of the matrix.
The enhanced structured approximation  improves the accuracy of the reduced basis method 
as justified by numerical simulations.
Moreover, it provides guaranteed upper and lower error bounds for 
exact eigenvalues of the BSE problem.
Section \ref{sec:Iter_Struct_BSE} describes structural iterative solvers for 
the central part of the spectrum in the simplified auxiliary problem.
In this way, the rank-structured approximation to the matrix inverse is constructed
based on the Sherman-Morrison formula.
%We propose to enhance the subspace iteration by use of the solution of TDA approximation as the initial guess.
% Further, we show that the diagonal
% plus low-rank decomposition of can be can be gainfully supplemented by the small-size
% block-truncation in the static screen interaction sub-matrix. 
Section \ref{sec:QTT_BSE} discusses the benefits of structured iterative solver 
based on the QTT tensor approximation of vectors and matrices in the framework
of ALS-type subspace iterations in block-QTT format. In particular, we 
present and analyze numerically the algorithm for solving the BSE problem 
in  $\mathcal{O}(\log(N_o)N_o^2)$ complexity scaling,
where $N_o \ll N_{b}$ denotes the number of occupied molecular orbitals.
Conclusions underline the main results and outlook directions for forthcoming work.

\section{Reduced basis approach to BSE problem revisited}\label{sec:RBA_BSE}

The construction of the BSE matrix includes  
computations of several auxiliary quantities \cite{ReTouSa1:13,BeKhKh_BSE:15} 
represented in terms of energy spectrum $\varepsilon_j$, $j=1,...,N_b$, and
the two-electron integrals (TEI) matrix projected onto the Hartree-Fock molecular orbital basis, 
$$
{ V}=[v_{ia,jb}]\quad a, b \in {\cal I}_{v}:=\{N_{o}+1,\ldots,N_{b}\},
\quad i,j\in {\cal I}_{o}:=\{1,\ldots,N_{o}\},
$$
where $N_b$ is the number of GTO basis functions and 
$N_{o}$ denotes the number of occupied orbitals 
(see \cite{VeKhorMP2:13,BeKhKh_BSE:15} for more details).

The $2\times 2$-block matrix representation of the Bethe-Salpeter equation % in the $(ov,vo)$ subspace 
reads as the following eigenvalue problem determining  the excitation energies $\omega_n$:
\begin{equation} \label{eqn:BSE-GW1}
   F
 \begin{pmatrix}
 {\bf x}_n\\
{\bf y}_n\\
\end{pmatrix}
\equiv
\begin{pmatrix}
{ A}   &  { B} \\
{ B}^\ast  &  { A}^\ast  \\
\end{pmatrix}
\begin{pmatrix}
 {\bf x}_n\\
{\bf y}_n\\
\end{pmatrix}
= \omega_n
\begin{pmatrix}
{I}   &  { 0} \\
{ 0}  &  { -I}  \\
\end{pmatrix}
\begin{pmatrix}
 {\bf x}_n\\ {\bf y}_n\\
\end{pmatrix},
\end{equation}
where the matrix blocks of size $N_{ov} \times N_{ov}$, with $N_{ov}=N_{o}(N_{b}-N_{o})$,
are defined by
\begin{equation}
 A=  \boldsymbol{\Delta \varepsilon} + V - \overline{W},\quad
 B= {V} - \widetilde{W}.
\label{eq:AB_ex}
\end{equation}
Here, the diagonal part is given by the ''energy'' matrix 
\[
\boldsymbol{\Delta \varepsilon}=[\Delta \varepsilon_{ia,jb}]\in \mathbb{R}^{N_{ov}\times N_{ov}}:
\quad \Delta \varepsilon_{ia,jb}= (\varepsilon_a - \varepsilon_i)\delta_{ij}\delta_{ab} ,
\]
that can be represented in the Kronecker product form
\[
\boldsymbol{\Delta \varepsilon} =
I_{o}\otimes \mbox{diag}\{\varepsilon_a: a\in {\cal I}_v\} -
\mbox{diag}\{\varepsilon_i: i\in {\cal I}_o\}\otimes I_{v},
\]
where $I_o$ and $I_v$ are the identity matrices on respective index sets. 
Matrices $\widetilde{W}$ and $\overline{W}$ are obtained by certain transformations of
the matrix $V$.

The matrices $V$ and $\widetilde{W}$ are proven to have small
$\epsilon$-rank (see \cite{VeKhBoKhSchn:12,VeKhorMP2:13} and \cite{BeKhKh_BSE:15}, respectively).
In particular, there holds 
\begin{equation} \label{eqn:L_V_factor}
 V \approx L_V L_V^T,\quad L_V\in \mathbb{R}^{N_{ov} \times R_V}, \quad R_V \leq R_B,
\end{equation}
with the rank estimates $R_V =R_V(\varepsilon) =\mathcal{O}(N_b |\log \varepsilon |)$, and 
$rank(\widetilde{W})\leq rank(V)$.
It was found that the matrix $\overline{W}$ can be approximated by the low-rank substitute only
up to the limited precision $\epsilon$, so that computationally unexpensive approach
to get rid of this limitation may be the rank approximation 
with the constraints $rank(\overline{W})\leq rank(V)$, see \cite{BeKhKh_BSE:15}.

Matrices in the form (\ref{eqn:BSE-GW1}) are called as the $J$-symmetric, see 
\cite{BeFaYa:15} for implications on the algebraic properties of the BSE matrix.
Solutions of equation (\ref{eqn:BSE-GW1}) come in pairs: excitation energies $\omega_n$
with eigenvectors $({\bf x}_n,{\bf y}_n)$, and de-excitation energies
$-\omega_n$ with eigenvectors $({\bf x}_n^\ast,{\bf y}_n^\ast)$.
The spectral problem (\ref{eqn:BSE-GW1}) can be rewritten  in the equivalent form
\begin{equation} \label{eqn:BSE-F1}
 %\Delta_1 = \frac{1}{h}
 F_1
 \begin{pmatrix}
 {\bf x}_n\\
{\bf y}_n\\
\end{pmatrix}
\equiv
\begin{pmatrix}
{ A}   &  { B} \\
-{ B}^\ast  &  -{ A}^\ast  \\
\end{pmatrix}
\begin{pmatrix}
 {\bf x}_n\\
{\bf y}_n\\
\end{pmatrix}
= \omega_n
% \begin{pmatrix}
% {I}   &  { 0} \\
% { 0}  &  { -I}  \\
% \end{pmatrix}
\begin{pmatrix}
 {\bf x}_n\\ {\bf y}_n\\
\end{pmatrix}.
\end{equation}

The dimension of the matrix in (\ref{eqn:BSE-GW1}) is $2 N_o N_v \times 2 N_o N_v$, 
where $N_o$ and $N_v$ denote the number of occupied and virtual orbitals, respectively.
In general, $N_o N_v$ is asymptotically of the order of  $\mathcal{O}(N_b^2)$, 
i.e. the spectral problem (\ref{eqn:BSE-GW1}) may become computationally extensive
even for moderate size molecules, say for $N_b \approx 100$.
Indeed, the direct eigenvalue solver for (\ref{eqn:BSE-GW1}) (full diagonalization) 
appears to be infeasible due to $\mathcal{O}(N_b^{6})$ complexity scaling.

The main idea of the {\it reduced basis approach} introduced in \cite{BeKhKh_BSE:15} 
can be described as follows.
Instead of solving the partial eigenvalue problem for finding of, say,
$m_0$ eigenpairs satisfying equation (\ref{eqn:BSE-F1}), we first solve the slightly
simplified auxiliary spectral problem with a modified matrix $F_0$.  
The approximation $F_0$ is obtained from $F_1$ by using low-rank approximation of matrices 
\begin{equation} \label{eqn:RankApprW}
\overline{W} \mapsto \overline{W}_r = L_W L_W^\top, \quad \mbox{and} 
\quad  \widetilde{W} \mapsto \widetilde{W}_r = Y Z^\top
\end{equation}
%$\overline{W} \mapsto \overline{W}_r$ and $\widetilde{W} \mapsto \widetilde{W}_r$ in
in the matrix blocks $A$ and $B$, respectively, i.e., $A$ and $B$ are replaced by
\begin{equation} \label{eqn:BSE-Reduce_Block}
A   \mapsto A_0:=\boldsymbol{\Delta \varepsilon} + V -  \overline{W}_r \quad \mbox{and} \quad
B \mapsto B_0: = {V} -  \widetilde{W}_r,
\end{equation}
where we assume for simplicity $rank(\overline{W}_r)\leq r$ and $rank(\widetilde{W}_r)\leq r$. 
Here we take into account that the matrix $V$, precomputed by tensor-based Hartree-Fock 
solver \cite{VeKhorTromsoe:15}, is already represented in the low-rank format (\ref{eqn:L_V_factor})
inherited from the Cholesky decomposition  of the TEI matrix $B$, see \cite{VeKhorMP2:13,BeKhKh_BSE:15}.

The modified auxiliary problem reads
\begin{equation} \label{eqn:BSE-Reduced}
 F_0
 \begin{pmatrix}
 {\bf u}_n\\
{\bf v}_n\\
\end{pmatrix}\equiv
\begin{pmatrix}
{ A}_0   &  { B}_0 \\
-{B}_0^\ast  &  -{A}_0^\ast  \\
\end{pmatrix}
\begin{pmatrix}
 {\bf u}_n\\
{\bf v}_n\\
\end{pmatrix}
= \lambda_n
% \begin{pmatrix}
% {I}   &  { 0} \\
% { 0}  &  { -I}  \\
% \end{pmatrix}
\begin{pmatrix}
 {\bf u}_n\\ {\bf v}_n\\
\end{pmatrix}.
\end{equation}
This eigenvalue problem is a simplification of  (\ref{eqn:BSE-F1}), since
now the matrix blocks $A_0$ and $B_0$, defined in (\ref{eqn:BSE-Reduce_Block}), 
are composed of diagonal and low-rank matrices, see Figures \ref{fig:Matr_A0} and \ref{fig:Matr_B0}
illustrating the data sparse structure of these matrix blocks.
\begin{figure}[htbp]
\centering
\includegraphics[width=11.0cm]{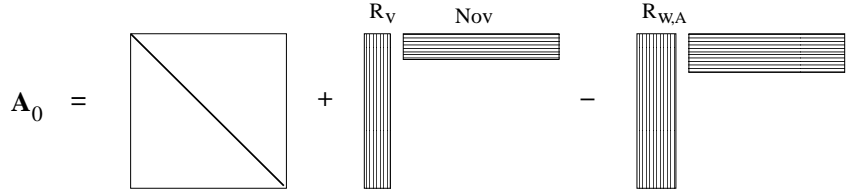}\quad 
\caption{\small Diagonal plus low-rank structure of the matrix $A_0$.}
\label{fig:Matr_A0}
\end{figure}

Having computed the set of eigenpairs $\{(\lambda_n, \psi_n)=(\lambda_n,({\bf u}_n,{\bf v}_n)^T)\}$, 
corresponding to $m_0 $ nearest to zero eigenvalues
(middle part of the spectrum) of the modified problem (\ref{eqn:BSE-Reduced}),
we solve the full eigenvalue
problem  for the reduced matrix (reduced model) obtained by projection of the initial equation onto the
problem adapted small basis set $\{\psi_n\}_{n=1}^{m_0}$ of size $m_0$.
\begin{figure}[htbp]
\centering
\includegraphics[width=8.0cm]{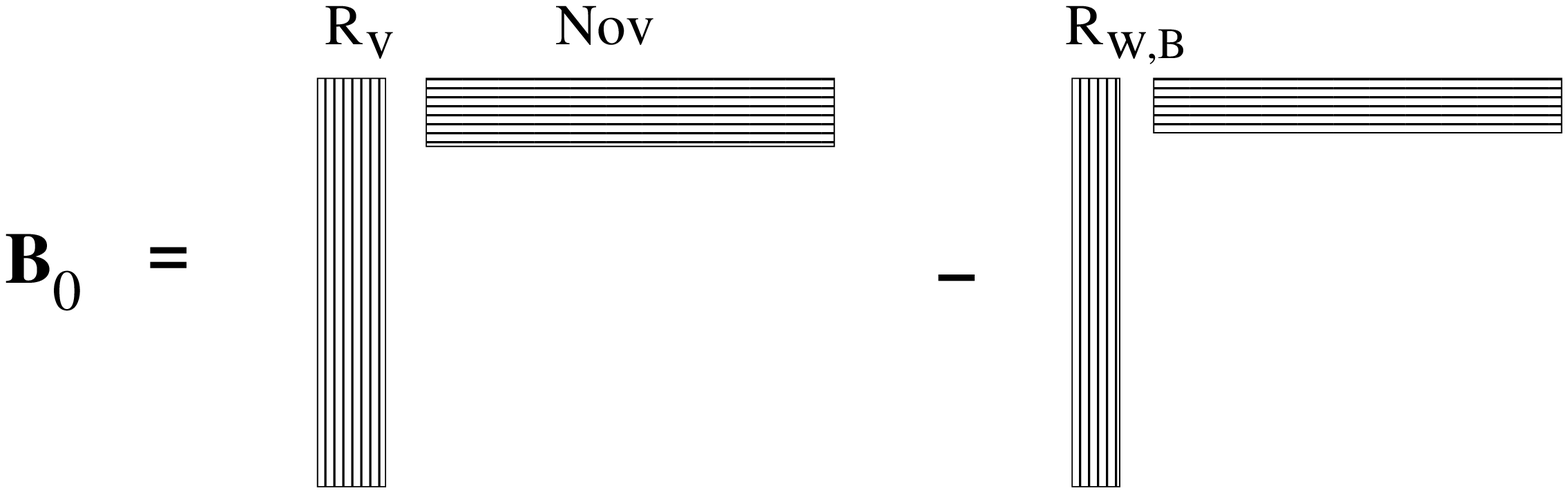}
\caption{\small Low-rank structure of the matrix $B_0$.}
\label{fig:Matr_B0}
\end{figure}

Define a matrix $G_1 = [\psi_1,...,\psi_{m_0}] \in\mathbb{R}^{2N_{ov}\times m_0}$, whose columns 
span  eigenvectors of the reduced basis, 
compute the related Galerkin  and mass matrices by projection onto the reduced basis 
specified by the columns in $G_1$,
\[
M_1= G_1^T F_1 G_1\in \mathbb{R}^{m_0\times m_0}, \quad S_1 = G_1^T G_1 \in \mathbb{R}^{m_0\times m_0},
\]
and then solve the reduced generalized eigenvalue problem of small size $m_0\times m_0$,
\begin{equation} \label{eqn:BSE-Red-Galerk}
 M_1 {\bf q}_n = \gamma_n S_1 {\bf q}_n, \quad {\bf q}_n \in \mathbb{R}^{m_0 }.
\end{equation}
%\textcolor{red}{$Y$ was used as a low-rank factor of $\tilde{W}$. Replace by $\mathbf{\hat u}_n$ or something?}
The portion of $m_0$  eigenvalues $\gamma_n$, %$n=1,\ldots,m_0$,
is thought to be very close to the lowers  excitation energies
$\omega_n$, ($n=1,\ldots,m_0$) in the initial spectral  problem (\ref{eqn:BSE-GW1}).

The so-called Tamm-Dancoff approximation (TDA) simplifies the equation (\ref{eqn:BSE-F1}) 
to a standard Hermitian eigenvalue problem
\begin{equation} \label{eqn:BSE-Tamm-Danc}
A {\bf x}_n = \mu_n {\bf x}_n, \quad {\bf x}_n \in \mathbb{R}^{N_{ov}}
\quad {A} \in \mathbb{R}^{N_{ov}\times N_{ov}}
\end{equation}
with the factor two smaller matrix size $N_{ov}$. 
The reduced basis approach via low-rank approximation
can be applied directly to the TDA equation, 
such that the simplified auxiliary problem reads
$$
A_0 {\bf u}= \lambda_n {\bf u},
$$
where we are interested in finding $m_0$ smallest eigenvalues.

Extensive numerical tests confirm the efficiency of the reduced model approach 
applied to both TDA and BSE problems for a number of single molecules,
as well as to chain type systems \cite{BeKhKh_BSE:15}.

Though the auxiliary eigenvalue equation (\ref{eqn:BSE-Reduce_Block}), (\ref{eqn:BSE-Reduced})
is much simpler than  (\ref{eqn:BSE-F1}), the  computation of dozens of eigenvectors
in (\ref{eqn:BSE-Reduced}) corresponding to the middle part of the spectrum remains 
to be the challenging numerical task since the traditional algebraic solvers 
%\textcolor{blue}
{often converge slowly. As a remedy, one can perform}
matrix-vector operations with the inverse matrix $A_0^{-1}$ or $F_0^{-1}$. 
The efficient construction and implementation of the structured matrix inverse 
$A_0^{-1}$ and $F_0^{-1}$ will be addressed in section \ref{sec:Iter_Struct_BSE}.

\section{Approximating matrix $\overline{W}$  in reduced-block format}
\label{sec:Iter_RedBlock_BSE}

Taking into account limitations of the low-rank decomposition to the 
static screen interaction matrix $\overline{W}$, in what follows, we introduce the 
alternative way to the data-sparse approximation of this matrix based on its restriction to 
a smaller-size active sub-matrix. 
\begin{figure}[htbp]
\centering
\includegraphics[width=6.0cm]{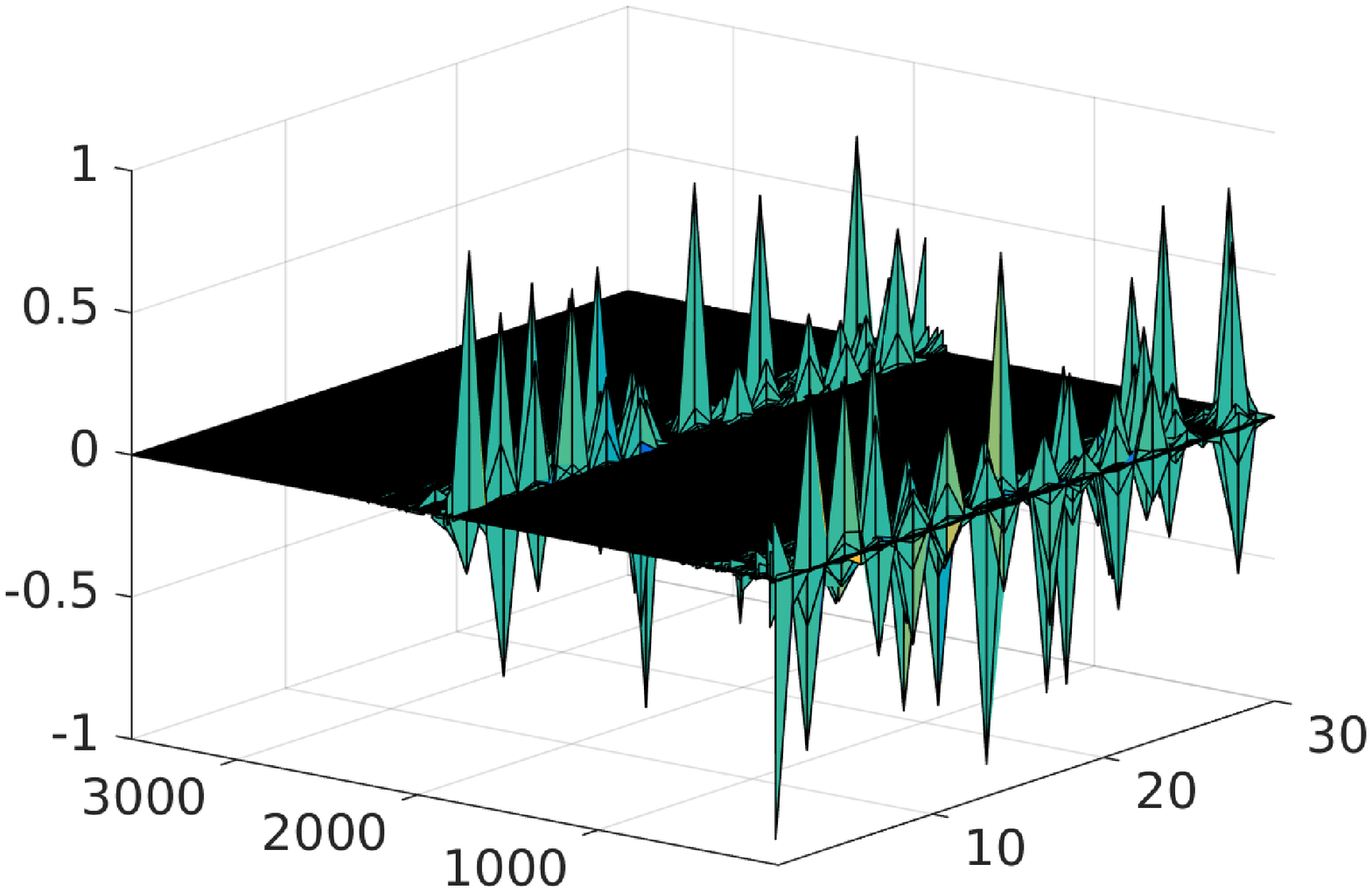}\quad\quad
\includegraphics[width=6.0cm]{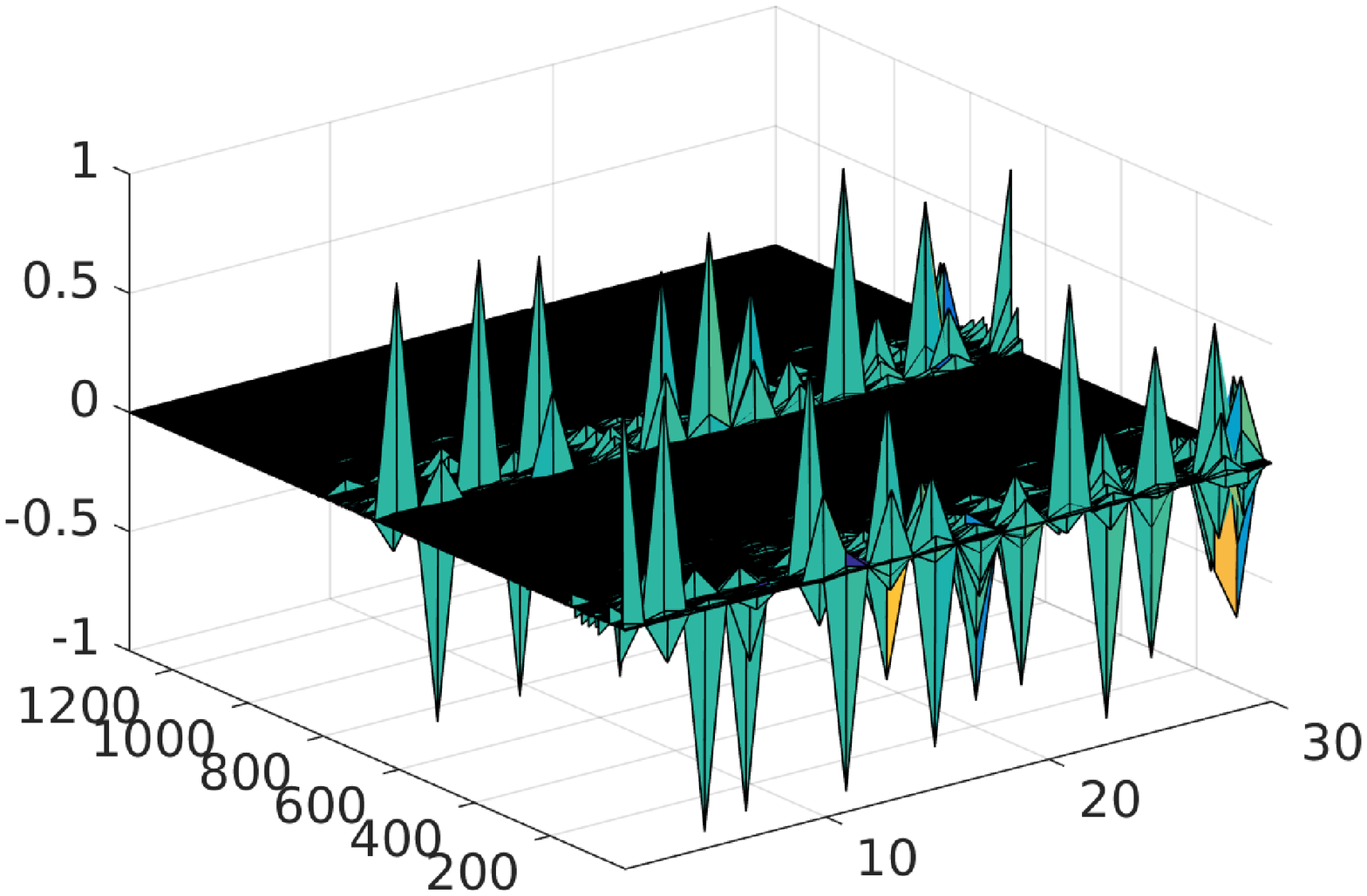} 
\caption{\small Visualizing first $m_0$ BSE eigenvectors  for H$_{32}$ chain and N$_2$H$_4$ 
molecule (right).}
\label{fig:EIGs_vectors}
\end{figure}

This approach is motivated by the numerical consideration (observed for all molecular systems 
considered so far) that eigenvectors corresponding to the central part 
of the spectrum have dominating components supported by rather small part 
of the full index set of size $2 \, N_{ov}$, see Figure \ref{fig:EIGs_vectors} for $m_0=30$. 
Indeed, their effective support is compactly located
at the first ``active'' indexes $\{1,...,N_W\}$ and $\{N_{ov}+1,...,N_W\}$ in the 
respective blocks, where $N_W \ll N_{ov}$.
\begin{figure}[htbp]
\centering
\includegraphics[width=11.0cm]{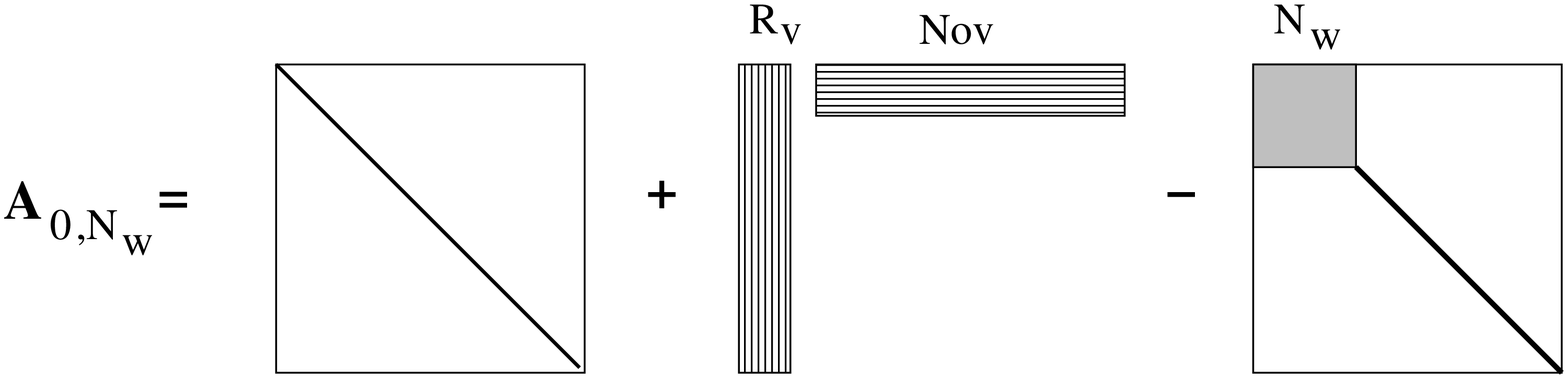} 
\caption{\small Diagonal plus low-rank plus reduced-block structure of the matrix $A_{0,N_W}$.}
\label{fig:Matr_AN}
\end{figure}

We define the selected sub-matrix $\overline{W}_{b}$ in $\overline{W}$, 
by keeping the balance between the storage size for the  
active sub-block $\overline{W}_{b}$ and the storage for the matrix $V$. 
Since the storage and numerical complexity of the rank-$R_V$ matrix $V$ is bounded by 
$2\, R_V \, N_{ov}$, we control the size of restricted $N_W\times N_W$ block 
$\overline{W}_{b}$ by the relation
\begin{equation} \label{eqn:NW_size}
 N_W= C_W \sqrt{2\; R_V \; N_{ov}},
\end{equation}
where the constant $C_W\approx 1$ is close to $1$. 
The approximation error introduced due to the
corresponding matrix truncation can be controlled by the choice of constant $C_W$. 

Keeping the diagonal in the matrix $\overline{W}$ unchanged, we define the simplified  
matrix by $\overline{W} \mapsto W_{N_W}\in \mathbb{R}^{N_{ov}\times N_{ov}}$, where
\begin{equation}
 W_{N_W}(i,j)=\left\{\begin{array}{ll}\overline{W}(i,j), &  i,j\leq N_W \quad \mbox{or} \quad i=j, \quad \mbox{and} \\
0 &  \mbox{otherwise}. \end{array}\right.
\end{equation}
The simplified matrix $A_{N_W}$ is then given by
\begin{equation} \label{eqn:auxil_block_A}
 A \mapsto A_{N_W}:=  \boldsymbol{\Delta \varepsilon} + V  - W_{N_W}, 
\end{equation}
while the modified block $B_0$ remains the same as in (\ref{eqn:BSE-Reduce_Block}).
The corresponding structure of the simplified matrix $A_{N_W}$
is illustrated in Figure \ref{fig:Matr_AN}.
\begin{table}[hbp]
 \begin{center}
 \begin{tabular}
[c]{|c|c|c|c|c| }%
\hline
  $C_W$ $\setminus$ $\varepsilon$   & $0.2 $            & $0.1$            & $0.05$  &
$0.01$    \\
 \hline
  $0.8$             &   $-0.09$; $0.006$ $(148)$  &  $-0.03$; $0.04$ $ (213)$   & 
$-0.008$; $0.014$ $ (284)$   &  $-0.005$; $0.0025$ $ (406)$ \\
 \hline
   $1.0$         &     $-0.1$; $0.05 $ $(185)$   &  $-0.036$; $0.03$ $ (266)$  & 
$-0.015$; $0.0076$ $(355)$ &  $-0.008$; $0.0003$ $(507)$  \\
   \hline
  $1.2$         &     $-0.1$; $0.05 $ $(222)$   &  $-0.04$; $0.02$ $ (320)$   &  
$-0.017$; $0.0038$ $ $ $(426)$  & $N_W=N_{ov}$ \\
   \hline
\end{tabular}
\end{center}
\caption{N$_2$H$_4$: Errors $\overline{\lambda}_1 - \omega_1$; $\overline{\gamma}_1 - \omega_1 $ 
  (in eV),  vs. $\varepsilon$ and $C_W$;  $N_W$ is given in brackets.}
\label{tab:err_N2H4_block}
\end{table}

This construction guaranties that the storage and matrix-vector multiplication complexity
for the simplified matrix block $A_{N_W}$ remains of the same order 
as that for the matrix $V$ characterized by low $\epsilon$-rank.

We modify the auxiliary matrix $F_0\mapsto \overline{F}_0$ in (\ref{eqn:BSE-Reduced}) 
by replacing $A_0 \mapsto A_{N_W}$, which leads to the corrections of 
the eigenvalues $\lambda_n \mapsto \overline{\lambda}_n$ and eigenvectors
$G_1 \mapsto \overline{G}_1=[\psi_1,...,\psi_{m_0}]\in \mathbb{R}^{2N_{ov}\times m_0}$
in the simplified problem, $\overline{F}_0 \psi_n = \overline{\lambda}_n \psi_n$. 
The corresponding eigenvalues $\overline{\gamma}_n$ of the modified reduced system
of the type (\ref{eqn:BSE-Red-Galerk}), specified by the Galerkin matrices
\[
\overline{M}_1= \overline{G}_1^T F_1 \overline{G}_1, 
\quad \overline{S}_1 = \overline{G}_1^T \overline{G}_1 \in \mathbb{R}^{m_0\times m_0},
\]
solve the eigenvalue problem
\begin{equation} \label{eqn:reduced_block_M1}
 \overline{M}_1 \mathbf{q}_n = \overline{\gamma}_n  \overline{S}_1 \mathbf{q}_n, \quad \mathbf{q}_n \in \mathbb{R}^{m_0}.
\end{equation}

Numerical examples below illustrate the approximation error vs. the rank truncation parameter 
$\epsilon>0$ in the reduced basis method characterized 
by the choice of the constant $C_W$ in the simplified matrix $A_{N_W}$ 
described in (\ref{eqn:auxil_block_A}). Spectral data and errors are given in eV.
 \begin{table}[htbp]
 \begin{center}
 \begin{tabular}
[c]{|c|c|c|c|c| }%
\hline
  $C_W$ $\setminus$ $\varepsilon$   & $0.2 $            & $0.1$             & $0.05$   
         & $0.01$    \\
 \hline
  $0.8$           &   $-0.23$; $0.13 $ $(131)$  &  $-0.054$; $0.08 $ $(157)$   & 
$-0.047$; $0.06$ $(168)$   &  $-0.006$; $0.02$ $(200)$ \\
 \hline
   $1.0$       &     $-0.28$; $0.06 $ $(164)$  &  $-0.1$; $0.01 $ $(196)$   & 
$-0.073$; $0.015 $ $(210$ &  $-0.005$; $0.02 $ $(250)$  \\
   \hline
  $1.2$         &     $-0.31$; $0.01 $ $(197)$   &  $-0.1$; $0.01$ $(236)$    & 
$-0.074$; $0.013$ $(251)$  & $-0.001$; $0.005$ $(301)$ \\
   \hline
\end{tabular}
\end{center}
\caption{H$_{16}$ chain: Errors $\overline{\lambda}_1 - \omega_1$; $\overline{\gamma}_1 - \omega_1 $ 
  (in eV),  vs. $\varepsilon$ and $C_W$; $N_W$ is given in brackets. }
\label{tab:err_H16_block}
\end{table}

Tables \ref{tab:err_N2H4_block} (N$_2$H$_4$ molecule) and Table \ref{tab:err_H16_block} 
(H$_{16}$ chain) demonstrate the numerical errors 
$\overline{\lambda}_1 - \omega_1$ and $\overline{\gamma}_1 - \omega_1$ for the 
minimal BSE eigenvalue $\omega_1$
indicating the two-sided error estimates addressed in Remark \ref{rem:bounds_4_Omega} below.
\begin{figure}[htbp]
\centering
\includegraphics[width=7.6cm]{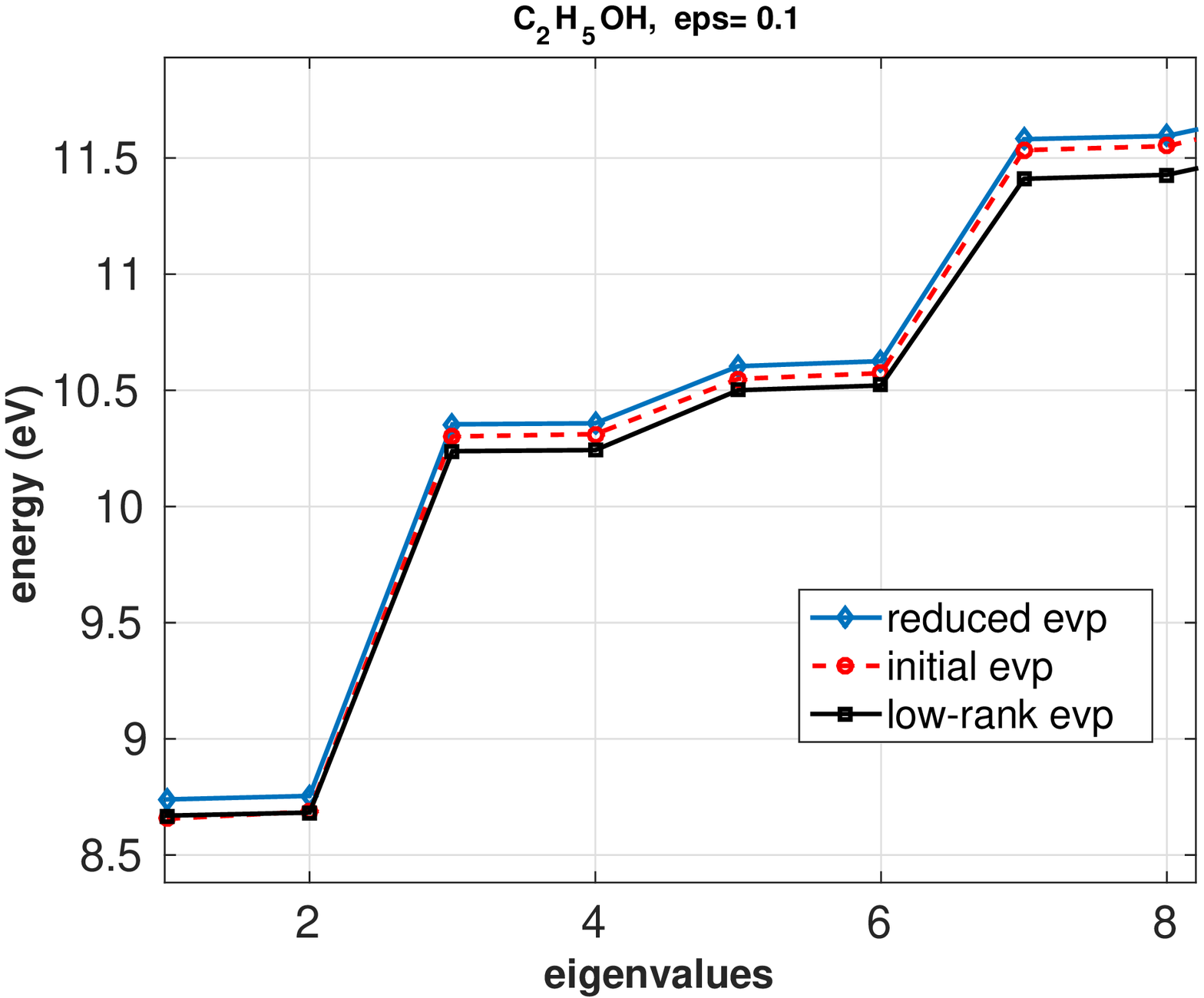}
\includegraphics[width=7.6cm]{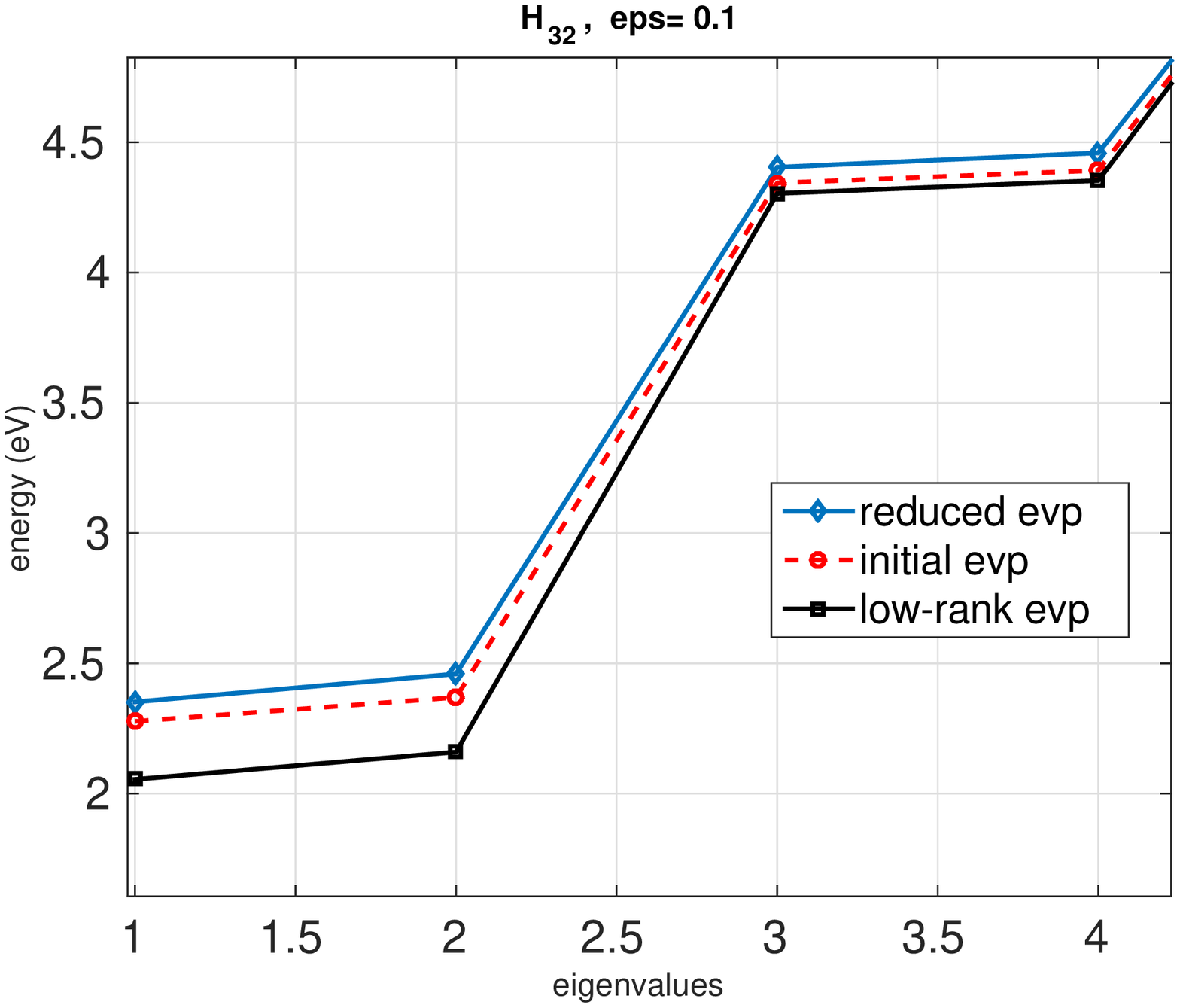} 
\caption{\small Two-sided bounds for the BSE excitation energies for C$_2$H$_5$OH and H$_{32}$ chain.}
\label{fig:EIGs_Bounds}
\end{figure}

\begin{remark}\label{rem:bounds_4_Omega}
It is worth to note that numerical results indicate the important property observed for all 
molecular systems tested so far: the close to zero
eigenvalues $\overline{\lambda}_k$ and $\overline{\gamma}_k$ provide the lower and 
upper bounds for the exact BSE eigenvalues $\omega_k$, i.e.
\[
 \overline{\lambda}_k \leq  \omega_k \leq \overline{\gamma}_k, \quad k=1,2,..., m_0.
\]
\end{remark}

Figure \ref{fig:EIGs_Bounds} demonstrates the two-sided error estimates declared in 
Remark \ref{rem:bounds_4_Omega}.
Here the ``black'' line represents the eigenvalues for the auxiliary problem of the 
type (\ref{eqn:BSE-Reduced}),
but with the modified matrix $\overline{F}_0$, while the blue line represents the eigenvalues of the 
reduced equation (\ref{eqn:reduced_block_M1})  of the type (\ref{eqn:BSE-Red-Galerk}) 
with the Galerkin matrices $\overline{M}_1$ and $\overline{S}_1$.

 Figure \ref{fig:EIGs_Bounds2} represents examples of upper and lower bounds 
for the whole sets of $m_0$ eigenvalues. 

\begin{figure}[htbp]
\centering
\includegraphics[width=7.6cm]{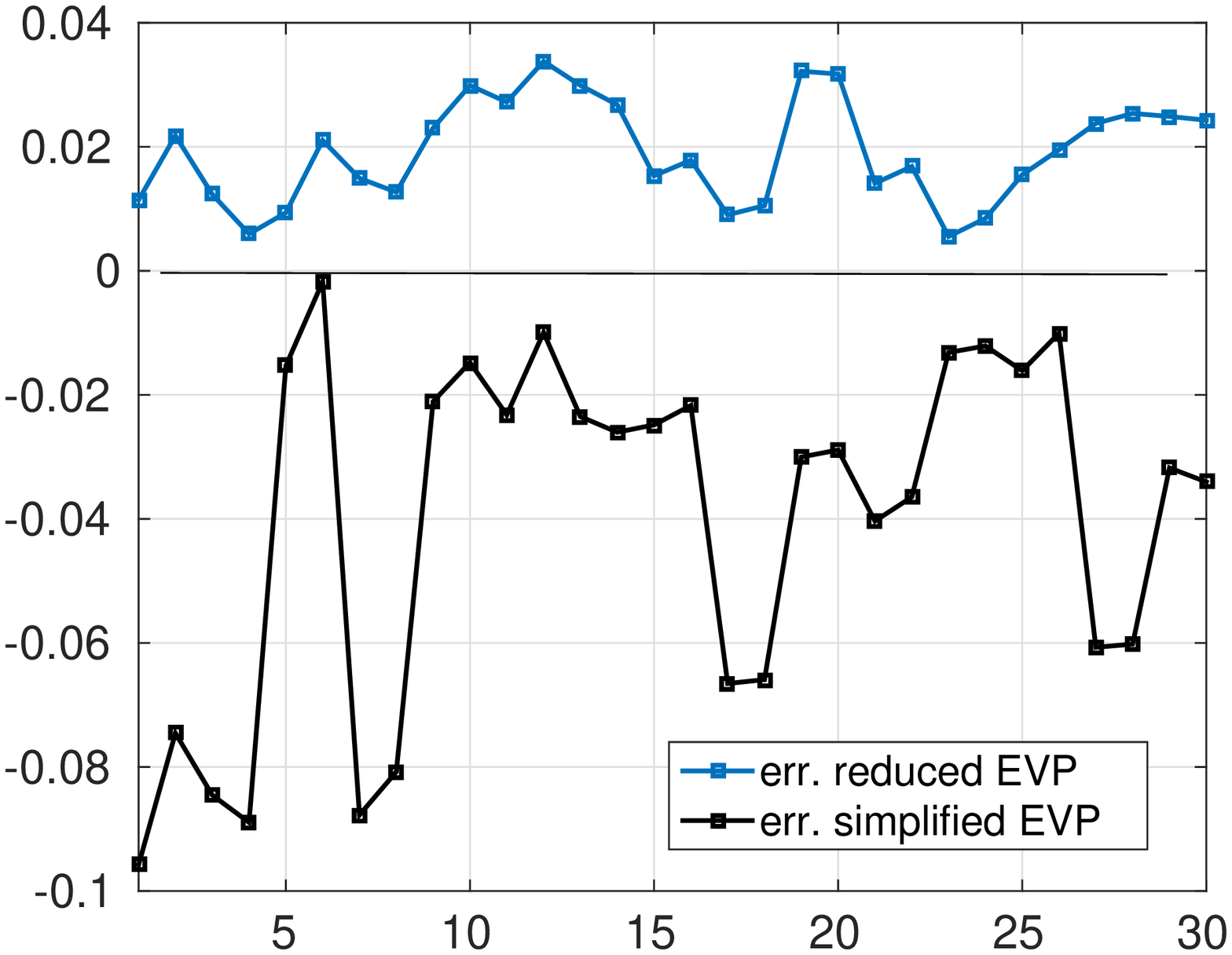}
\includegraphics[width=7.6cm]{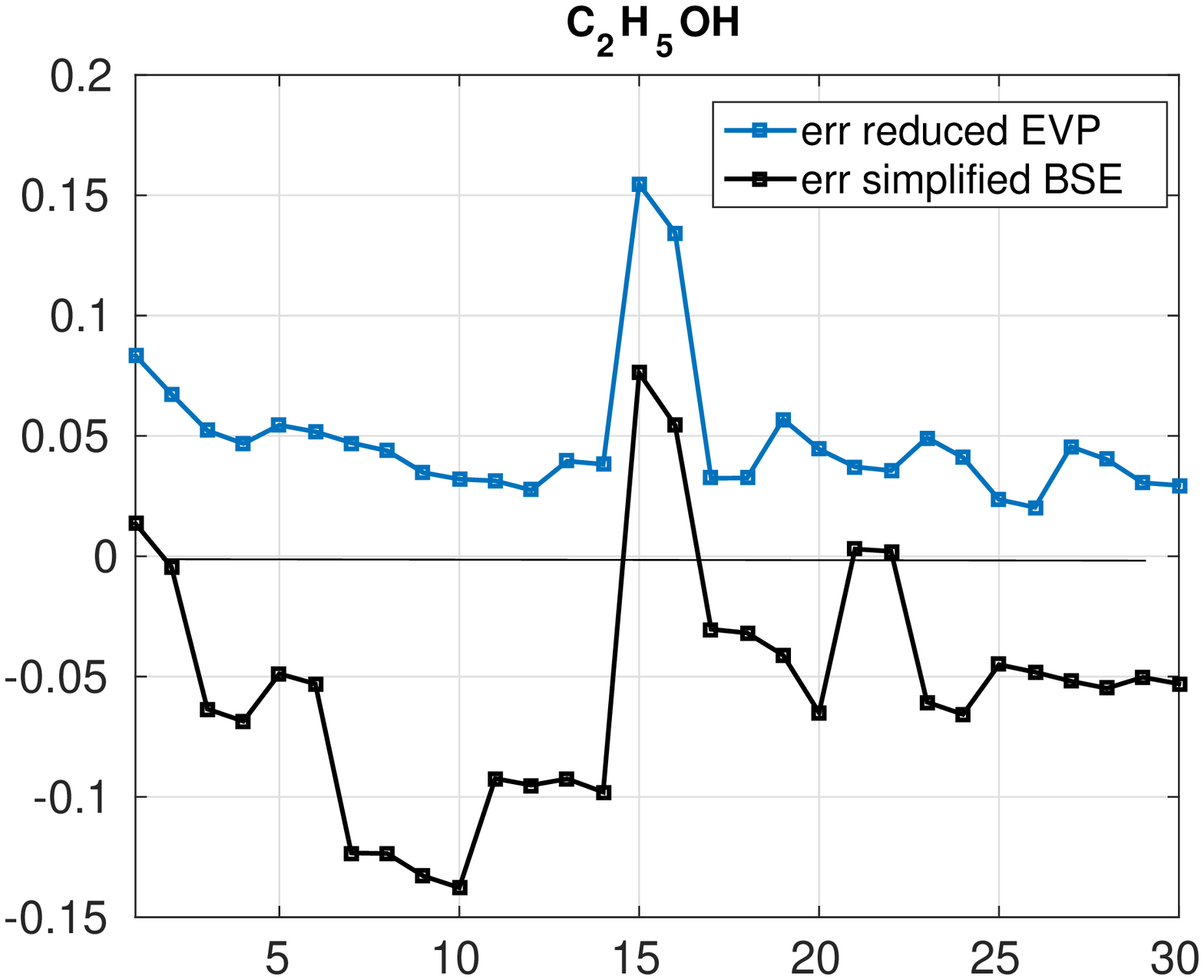} 
\caption{\small The errors  (in eV)  for simplified and reduced BSE eigenvalues for  H$_{16}$ chain 
and C$_2$H$_5$OH molecule (right). Zero level designates the solution of the initial BSE problem.}
\label{fig:EIGs_Bounds2}
\end{figure}

We conclude that the reduced basis approach based on the modified auxiliary matrix $\overline{F}_0$
via reduced-block anzats (\ref{eqn:auxil_block_A}), provides significantly better approximations
$\overline{\gamma}_n$  than that for the initial system with the matrix $F_0$, 
which is noticeable already for $\overline{\lambda}_n$.

\section{Iterative solver for central part of the spectrum} 
\label{sec:Iter_Struct_BSE}

In this section we discuss the construction of iterative solver for partial eigenvalue 
problem in (\ref{eqn:BSE-Reduced}) focusing on rank-structured approximation of 
the matrix inverse $A_0^{-1}$ and $F_0^{-1}$,
further optimization of the sparsity pattern in $\overline{W}$ and on the choice 
of initial guess by using solutions of the TDA model.

\subsection{Inverse iteration for diagonal plus low-rank matrix} \label{ssec:Iter_BSE}

Iterative eigenvalue solvers, such as Lancosh or Jacobi-Davidson methods, are quite efficient 
in approximation of the dominant eigenvalues, 
but may suffer from slow convergence if applied for computation of the smallest or 
intermediate eigenvalues.
We are interested in both of these scenarios.
There are both positive and negative eigenvalues in \eqref{eqn:BSE-Reduced}, and we need 
the few ones with the smallest magnitude.
In the TDA model (\ref{eqn:BSE-Tamm-Danc}), we solve a symmetric positive definite problem
$A_0 {\bf u}= \lambda_n {\bf u}$, 
but again the smallest eigenvalues are required. 

In both cases, the remedy is to invert the system matrix,  so  that the eigenvalues of interest become dominant.
The Matlab interface to ARPACK (procedure \texttt{eigs}) assumes by default that the 
user-defined function solves a linear system with the matrix instead of multiplying it, 
when the smallest eigenvalues are requested.
In our case, we can implement this efficiently, since the matrix consists of an easily 
invertible part (diagonal), plus a low-rank correction, and hence we can use the 
Sherman-Morrison formula.

To shorten the notation, we set up  the rank-$r$ decompositions following (\ref{eqn:RankApprW}), 
$\overline{W}_r = L_{W} L_{W}^\top$, 
$\widetilde{W}_r = Y Z^\top$, and define
\begin{equation}\label{lem:LowRankA0B0}
\begin{array}{lll}
 A_0  &= \boldsymbol{\Delta\varepsilon} +  PQ^\top, &\quad P = \begin{bmatrix}L_V & L_{W}
 \end{bmatrix}, \quad Q = \begin{bmatrix}L_V & -L_{W}\end{bmatrix}, \\[0.5em]
 B_0  &= \Phi \Psi^\top, &\quad \Phi = \begin{bmatrix}L_V & Y\end{bmatrix}, \quad \Psi = 
 \begin{bmatrix}L_V & -Z\end{bmatrix}.
\end{array}
\end{equation}
taking into account (\ref{eqn:L_V_factor}).

\begin{algorithm}[htb]
\caption{Precomputation of parts of $A_0^{-1}$ and $F_0^{-1}$}
\label{alg:sherm}
\begin{algorithmic}[1]
 \REQUIRE $\boldsymbol{\Delta\varepsilon}$ and low-rank factors of $V$, $\overline{W}_r$, 
 $\widetilde{W}_r$ $\eqref{eqn:RankApprW}$.
 \STATE Assemble $P = \begin{bmatrix}L_V & L_{W}\end{bmatrix}$, $Q = \begin{bmatrix}L_V & -L_{W}
 \end{bmatrix}$, $\Phi = \begin{bmatrix}L_V & Y\end{bmatrix}$,  $\Psi = \begin{bmatrix}L_V & -Z
 \end{bmatrix}$.
 \STATE Compute $P_{\varepsilon} = \boldsymbol{\Delta\varepsilon}^{-1}P$, $Q_{\varepsilon} = 
 \boldsymbol{\Delta\varepsilon}^{-1}Q$.
 \STATE Compute $K = (I+Q^\top P_{\varepsilon})^{-1}\in \mathbb{R}^{2r \times 2r}$.
 \STATE Compute $P_{\varepsilon K} = P_{\varepsilon} K$. \COMMENT{Enough for $A_0$}

  \vskip 0.5em
 \STATE Compute $\Phi_{\varepsilon} = \boldsymbol{\Delta\varepsilon}^{-1}\Phi$, 
 $\Psi_{\varepsilon} = \boldsymbol{\Delta\varepsilon}^{-1} \Psi$.
 \STATE Parts of $Q_S$: $\Phi_{\varepsilon P} = \Phi_{\varepsilon}^\top P$, 
 $\Phi_{\varepsilon Q} = Q^\top \Phi_{\varepsilon}$.
 \STATE Assemble $Q_{S \varepsilon} = \begin{bmatrix}Q_{\varepsilon} & \Psi_{\varepsilon} 
 \left(\Phi_{\varepsilon P} K \Phi_{\varepsilon Q} - \Phi^\top \Phi_{\varepsilon}\right)
 \end{bmatrix}$, $ P_{S \varepsilon} =  \begin{bmatrix}P_{\varepsilon} & \Psi_{\varepsilon}
 \end{bmatrix}$.
 \STATE Compute $K_S = (I+\begin{bmatrix}P & \Psi\end{bmatrix}^\top Q_{S \varepsilon})^{-1} 
 \in \mathbb{R}^{4r \times 4r}$
 \STATE Compute $Q_{S \varepsilon K} = Q_{S \varepsilon} K_S$. \COMMENT{For the Schur complement}

  \vskip 0.5em
 \STATE Compute $\Phi_{AB} = \boldsymbol{\Delta\varepsilon}^{-1}\Phi -  P_{\varepsilon K} 
 \left(Q_{\varepsilon}^\top \Phi\right)$. \COMMENT {For $A_0^{-1}B_0$}
\end{algorithmic}
\end{algorithm}

\begin{algorithm}[htb]
\caption{Solution of linear systems with $A_0$ and $F_0$}
\label{alg:sherm-app}
\begin{algorithmic}[1]
\REQUIRE Precomputed matrices $P_{\varepsilon K}, Q_{\varepsilon}, Q_{S \varepsilon K}, P_{S \varepsilon}, 
\Phi_{AB}$ from Alg. \ref{alg:sherm} and $\boldsymbol{\Delta\varepsilon},\Phi,\Psi$.
\ENSURE $\mathbf{\tilde z}=A_0^{-1}\mathbf{u}$ and 
$\begin{bmatrix}\mathbf{z} \\ \mathbf{y}\end{bmatrix}=F_0^{-1}\begin{bmatrix}\mathbf{u} \\ \mathbf{v}\end{bmatrix}$
\STATE Apply the TDA inverse as  $\mathbf{\tilde z} \equiv A_0^{-1}\mathbf{u} = 
\boldsymbol{\Delta\varepsilon}^{-1}\mathbf{u} -  P_{\varepsilon K} \left(Q_{\varepsilon}^\top \mathbf{u}\right)$.

\vskip 0.5em
\STATE Compute $\mathbf{\tilde y} = \mathbf{v} + \Psi \left(\Phi^\top \mathbf{\tilde z}\right)$ $\eqref{eqn:F0inv}$
\STATE Apply the Schur complement $\mathbf{y} \equiv S^{-1}\mathbf{\tilde y} = 
-\boldsymbol{\Delta\varepsilon}^{-1} \mathbf{\tilde y} + Q_{S \varepsilon K} 
 \left(P_{S \varepsilon}^\top \mathbf{\tilde y}\right)$.
\STATE Compute $\mathbf{z} = \mathbf{\tilde z} - \Phi_{AB} \left(\Psi^\top\mathbf{y}\right)$.
\end{algorithmic}
\end{algorithm}

\begin{lemma}\label{lem:EigStructCost} (Complexity of the diagonal plus low-rank approach).
Let the rank parameters in decompositions of $V$, $\overline{W}$ and $\widetilde{W}$ do not exceed $r$.
Then the rank structured representations of inverse matrices $A_0^{-1} $ and $F_0^{-1} $
can be precomputed with the overall cost $\mathcal{O}(N_{ov} r^2)$.
The complexity for each inversion $A_0^{-1} \mathbf{u}$ or $F_0^{-1} \mathbf{w}$ is bounded by 
$\mathcal{O}(N_{ov} r)$. 
\end{lemma}
% \textcolor{red}{90\% of people will not read the section called ``proof''. 
% However, there are important formulas for practical computations. 
% Therefore I write a separate Algorithm 2 and refer to the necessary Eqns. from there.}

\begin{proof}
We begin with the TDA model (\ref{eqn:BSE-Tamm-Danc}).
The Sherman-Morrison formula for $A_0$ in (\ref{lem:LowRankA0B0}) reads
\begin{equation}
 A_0^{-1} = \boldsymbol{\Delta\varepsilon}^{-1} - \boldsymbol{\Delta\varepsilon}^{-1} P 
 \left(I + Q^\top \boldsymbol{\Delta\varepsilon}^{-1} P\right)^{-1} Q^\top 
 \boldsymbol{\Delta\varepsilon}^{-1}.
\label{eqn:A_0_sherm}
\end{equation}
Here the inner $2r \times 2r$ matrix $K=\left(I + Q^\top \boldsymbol{\Delta\varepsilon}^{-1} P\right)^{-1}$ 
is small and can be computed explicitly at the expense  $\mathcal{O}(r^3 + r^2 N_{ov})$.
Hence matrix-vector product $A_0^{-1} \mathbf{u}_n$ requires multiplication by the diagonal 
matrix $\boldsymbol{\Delta\varepsilon}^{-1}$ and the low-rank matrix in the second summand.
This amounts to the overall cost $\mathcal{O}(N_{ov} r)$.

To invert $F_0$, we first derive its LU decomposition.
One can verify that 
\begin{equation}
F_0 = \begin{bmatrix} A_0 & B_0 \\ -B_0^\top & -A_0^\top \end{bmatrix} = 
\begin{bmatrix} A_0 & 0 \\ -B_0^\top & I \end{bmatrix} \begin{bmatrix} 
I & A_0^{-1} B_0 \\ 0 & S \end{bmatrix}, \qquad S = -A_0^\top + B_0^\top A_0^{-1} B_0.
\end{equation}
To solve a system $F_0 \begin{bmatrix}\mathbf{z} \\ \mathbf{y}\end{bmatrix} = 
\begin{bmatrix}\mathbf{u} \\ \mathbf{v}\end{bmatrix}$, 
we need one action of $A_0^{-1}$ and inverse of the Schur complement $S^{-1}$.
Indeed,
\begin{equation}\label{eqn:F0inv}
 \begin{split}
 \mathbf{\tilde z} & = A_0^{-1} \mathbf{u}, \quad \mathbf{\tilde y} = \mathbf{v} + B_0^\top \mathbf{\tilde z}, \\
 \mathbf{y} & = S^{-1} \mathbf{\tilde y}, \quad \mathbf{z} = \mathbf{\tilde z} - A_0^{-1} B_0 \mathbf{y}.
 \end{split}
\end{equation}
Note that $A_0^{-1} B_0$ is a low-rank matrix and can be precomputed in advance.
The action of $A_0^{-1}$ is given by \eqref{eqn:A_0_sherm}, so we address now the 
inversion of the Schur complement.

Plugging \eqref{eqn:A_0_sherm} into $S$, we obtain
$$
S = -\boldsymbol{\Delta\varepsilon} - QP^\top + \Psi \Phi^\top A_0^{-1} \Phi \Psi^\top 
= - (\boldsymbol{\Delta\varepsilon} + Q_S P_S^\top), % \quad \mbox{where}
$$
where
\begin{equation}\label{eqn:Qs_inSchur}
 Q_S = \begin{bmatrix} Q & \Psi  \left(\Phi^\top \boldsymbol{\Delta\varepsilon}^{-1} P K Q^\top 
 \boldsymbol{\Delta\varepsilon}^{-1} \Phi - \Phi^\top \boldsymbol{\Delta\varepsilon}^{-1} \Phi \right) 
 \end{bmatrix}, \quad P_S =  \begin{bmatrix}P & \Psi\end{bmatrix}.
\end{equation}
Therefore, 
\begin{equation}\label{eqn:Schur_lowR}
 S^{-1} = -\left(\boldsymbol{\Delta\varepsilon}^{-1} - 
 \boldsymbol{\Delta\varepsilon}^{-1} Q_S K_S P_S^\top \boldsymbol{\Delta\varepsilon}^{-1}\right), 
 \quad K_S = \left(I + P_S^\top \boldsymbol{\Delta\varepsilon}^{-1} Q_S\right)^{-1}.
\end{equation}
Keeping intermediate results in these calculations, we can trade off the memory against the CPU time.
The computational cost of (\ref{eqn:Qs_inSchur}) and then (\ref{eqn:Schur_lowR}) is again bounded by 
$\mathcal{O}(r^2 N_{ov})$, while the implementation of (\ref{eqn:F0inv}) takes $\mathcal{O}(r N_{ov})$ operations.
\end{proof}

%\textcolor{blue}
{Lemma \ref{lem:EigStructCost} indicates that for both BSE and TDA models
the asymptotic complexity for one iterative step is of the same order.
Precomputation of intermediate matrices is described in Algorithm \ref{alg:sherm}, 
and their use in the structured matrix inversion is shown in Algorithm \ref{alg:sherm-app}.}

Table \ref{tab:times_TDA} compares CPU times (sec) for full \texttt{eig} and rank-structured iteration for 
TDA problem (\ref{eqn:BSE-Tamm-Danc}) in Matlab implementation. 
Rank-truncation threshold is $\varepsilon = 0.1$, the number of computed eigenvalues is $m_0=30$.
%\textcolor{blue}
{Bottom line shows CPU times (sec) of the \texttt{eigs} procedure applied with the inverse 
matrix-vector product $A_0^{-1}\mathbf{u}$ using Algorithm \ref{alg:sherm-app} (marked by "inv").
The other lines show results of the corresponding algorithms which used traditional 
product $A_0\mathbf{u}$ ($A_0$ in the low-rank form). Notice that the results for Matlab 
version of LOBPCG by \cite{knyazev-lobpcg-2001} are presented for comparison.
We see that the inverse-based method is superior in all tests.}
\begin{table}[htbp]
 \begin{center}
 \begin{tabular}
[c]{|c|c|c|c|c|c|c|c| }%
\hline
Molecular syst.& H$_2$O & N$_2$H$_4$ & C$_2$H$_5$OH & H$_{32}$  & C$_2$H$_5$ NO$_2$ & H$_{48}$ & C$_3$H$_7$ NO$_2$ \\
         \hline\hline
 TDA  size & $180^2$   & $657^2$ & $1430^2$   &  $1792^2$    & $3000^2$  & $4032^2$ & $4488^2$   \\
\hline \hline
\texttt{eig}($A_0$)      & $0.02$   & $0.5$   & $4.3$      & $9.8$  & $37.6$ & $ 91$ &  $127.4$   \\
\hline
\texttt{lobpcg}($A_0$)  & $0.22$   & $0.6$    & $5.4$    & $2.77 $  & $18.2$ & $5.6$ &  $34.2$   \\
\hline
\texttt{eigs}($A_0$)      & $0.07$   & $0.29$   & $1.7$      & $0.49$  & $12$ & $2.7$ &  $21$   \\
\hline
\texttt{eigs}(inv($A_0$)) & $0.05$   & $0.08$   & $0.17$    & $0.11$  & $0.32$ & $0.34 $ &  $0.5$   \\
\hline
 \end{tabular}
\end{center}
\caption{Times (s) for eigensolvers applied to TDA matrix. }
\label{tab:times_TDA}
\end{table}
 \begin{remark} 
Notice that the initial guess for the subspace iteration applied to the full BSE can be constructed,
%\textcolor{blue}
{replicating the eigenvectors computed in the TDA model.}
It provides rather accurate approximation to the exact eigenvectors for the initial BSE system (\ref{eqn:BSE-F1}).
 In \cite{BeKhKh_BSE:15} it was shown numerically that the TDA approximation error 
 $|\mu_n - \omega_n|$ of the order of $10^{-2}$ eV is achieved 
 for compact and extended molecules presented in Table \ref{tab:times_TDA}.
 \end{remark}

Table \ref{tab:times_BSE} compares CPU times (sec) for EIG-solver and rank-structured EIGS-iteration  
applied to the full BSE problem (\ref{eqn:BSE-F1}). 
\begin{table}[htbp]
  \begin{center}
  \begin{tabular}
[c]{|c|c|c|c|c|c|c|c|}%
\hline
Molecular syst.        & H$_2$O    & N$_2$H$_4$  & C$_2$H$_5$OH & H$_{32}$    &C$_2$H$_5$ NO$_2$ & H$_{48}$ 
&  C$_3$H$_7$ NO$_2$ \\
  \hline
 $ N_{o}$, $N_{b}$& $5$, $41$ & $9$ , $82$  & $13$, $123$  & $16$, $128$ & $20$, $170$ & $24$, $192 $  
 & $24$, $211$   \\
\hline
BSE matrix size & $360^2$  &  $1314^2$ & $2860^2$    & $3584^2$  & $6000^2$  & $8064^2$ & $8976^2$   \\
\hline
\texttt{eig}($F_0$)  & $0.08$    & $4.2$      & $33.7$    & $68.1$  & $274$  & $649$ & $903$   \\
\hline
\texttt{eigs}($F_0$)  & $0.13$     & $0.28$   & $0.7$   & $0.77$  & $2.2$  & $2.3$ & $3.9$   \\
\hline
% \hline
% EIGS(inv(low-r($F_0$))) & $0.17$  & $0.27$   & $0.81$   & $0.74$  & $2.3$  & $2.4$ & $4.0$   \\
% \hline
\end{tabular}
\end{center}
\caption{Times (s) for the full BSE matrix $F_0$. }
\label{tab:times_BSE}
\end{table}
% \textcolor{red}{What is the difference between the last two rows?
% EIGS($F_0$) is always performed using the Sherman inversion.
% Since the CPU times are quite the same, I would guess those refer to the same method, and we can remove one row safely.}

\subsection{Inversion of the block-sparse matrices} \label{ssec:Iter_BSE_BlSp}

%\textcolor{blue}{
If ${W}_{N_W}$ is kept in the block-diagonal form as in (\ref{eqn:auxil_block_A}), 
its inversion remains still easy similar to the case (\ref{eqn:BSE-Reduce_Block}).
Basically, we can use the same Sherman-Morrison scheme from Algorithms \ref{alg:sherm}, and \ref{alg:sherm-app}.
To that end, we aggregate $\boldsymbol{\Delta\varepsilon}_W = 
\boldsymbol{\Delta\varepsilon} - {W}_{N_W}$, while in the low-rank factors the 
only $P=Q=L_V$ remains.
After that, all calculations in Algorithms \ref{alg:sherm} and \ref{alg:sherm-app} are repeated unchanged, replacing 
all $\boldsymbol{\Delta\varepsilon}$ by $\boldsymbol{\Delta\varepsilon}_W$, where the latter is 
now a block-diagonal matrix.

The particular modifications for the enhanced algorithm are as follows. 
Let us split
$\boldsymbol{\Delta\varepsilon} = \mathrm{blockdiag}(\boldsymbol{\Delta\varepsilon}_1, 
\boldsymbol{\Delta\varepsilon}_2)$,
where $\boldsymbol{\Delta\varepsilon}_1$ has the size $N_W$, and 
$\boldsymbol{\Delta\varepsilon}_2 \in \mathbb{R}^{N_W' \times N_W'}$ with $N_W'=N_{ov}-N_W$
represents the remaining values.
The same applies to $W_{N_W} = \mathrm{blockdiag}(W_b, \mathrm{diag}(w_2))$, 
where $w_2$ contains the elements on the diagonal of $W_{N_W}$ which do not belong to $W_b$.
Then the implementation of the matrix inverse
\begin{equation}\label{eqn:DepsInv}
\boldsymbol{\Delta\varepsilon}_W^{-1} = 
\mathrm{blockdiag}((\boldsymbol{\Delta\varepsilon}_1 - W_b)^{-1}, 
(\boldsymbol{\Delta\varepsilon}_2 - \mathrm{diag}(w_2))^{-1})
\end{equation}
requires inversion of an $N_W \times N_W$ dense matrix, and a diagonal matrix of size $N_W'=N_{ov}-N_W$.
Since $N_W$ is chosen small, the complexity of this operation is moderate.
Now all steps requiring multiplication with $\boldsymbol{\Delta\varepsilon}$ in  
Algorithm \ref{alg:sherm}--\ref{alg:sherm-app} can be substituted by (\ref{eqn:DepsInv}).
Numerical complexity of the new inversion scheme is estimated in the next lemma.

\begin{table}[htbp]
 \begin{center}
 \begin{tabular}
[c]{|c|c|c|c|c|c|c|c| }%
\hline
Molecular syst.& H$_2$O & N$_2$H$_4$ & C$_2$H$_5$OH & H$_{32}$  & C$_2$H$_5$ NO$_2$ & H$_{48}$ 
& C$_3$H$_7$ NO$_2$ \\
 \hline\hline
 TDA  size & $180^2$   & $657^2$ & $1430^2$   &  $1792^2$    & $3000^2$  & $4032^2$ & $4488^2$   \\
\hline
TDA: \texttt{eigs}(block($A_0$)) & $0.09$ & $0.33$   & $2.8$  & $0.77$  & $16.1$ & $3.0$ &  $30$   \\
\hline
TDA: \texttt{eigs}(block-inv($A_0$)) & $0.07$ & $0.09$   & $0.25$  & $0.77$    & $0.54$ & $3.0$ &  $1.0$   \\
\hline\hline
BSE: \texttt{eigs}(inv(block($F_0$)))  & $0.21$     & $0.37$   & $1.11$   & $1.10$  & $2.4$  & $2.92$ & $4.6$   \\
\hline
BSE vs. $\overline{F_0}$: $|\overline{\gamma}_1-\omega_1 | $ & $0.02$ & $0.03$ & $0.08$ & $0.07$  
& $0.05$ & $0.10$ &  $0.1$   \\
\hline
\hline
 \end{tabular}
\end{center}
\caption{Block-sparse matrices: times (s) for eigensolvers applied to TDA and BSE systems. 
 Bottom line shows the error (eV) for the case of block-sparse approximation to 
 the diagonal matrix block $A_0$, $\varepsilon = 0.1$.}
\label{tab:times_TDABSE_Blsp}
\end{table}

\begin{lemma}\label{lem:BlSp_Compl} (Complexity of the reduced-block algorithm).
 Suppose that the rank parameters in the decomposition of $V$ and $\widetilde{W}$ 
 do not exceed $r$ and  the block-size $N_W$ is chosen from the equation (\ref{eqn:NW_size}).
 
 Then the rank structured plus reduced-block representations of inverse matrices 
 $A_0^{-1} $ and $F_0^{-1} $
can be set up with the overall cost $\mathcal{O}(N_{ov}^{3/2} r^{3/2} + N_{ov} r^{2})$.
The complexity of each inversion $A_0^{-1} \mathbf{u}$ or $F_0^{-1} \mathbf{w}$ is bounded by $\mathcal{O}(N_{ov} r)$. 
\end{lemma}
\begin{proof}
Inversion of the $N_W \times N_W$ dense block in \eqref{eqn:DepsInv} requires $\mathcal{O}(N_W^3)$ operations.
Hence, the condition (\ref{eqn:NW_size}) ensures that the cost of setting up the matrix 
(\ref{eqn:DepsInv}) is bounded by $\mathcal{O}(N_{ov}^{3/2} r^{3/2})$. 
%\textcolor{blue}
{After that, multiplication of \eqref{eqn:DepsInv} by a $N_{ov} \times r$ matrix 
(e.g. in Line 2 of Alg. \ref{alg:sherm}) requires $\mathcal{O}(N_W^2 r + N_W'r) = \mathcal{O}(N_{ov}(r^2+r))$ operations.
In Alg. \ref{alg:sherm-app}, multiplication of \eqref{eqn:DepsInv} by a vector 
is performed with the $\mathcal{O}(N_W^2 + N_W') = \mathcal{O}(N_{ov}r)$ cost.
Complexity of the other steps is the same as in Lemma \ref{lem:EigStructCost}.
}
\end{proof}

Numerical illustrations for the enhanced data sparsity are presented in Table \ref{tab:times_TDABSE_Blsp}.

Notice that the performance of the low-rank 
and block-sparse solvers is comparable, but the second one provides the better sparsity and 
higher accuracy in eigenvalues, see \S\ref{sec:Iter_RedBlock_BSE}.
Remarkable that the most advanced version of the approach, based on the inverse iteration applied
to the diagonal plus low-rank plus reduced-block approximation, outperforms the
full eigenvalue solver on several orders of magnitude.

\section{Solving BSE spectral problem in the QTT format }\label{sec:QTT_BSE}

\subsection{Rank-structured representation of multi-dimensional tensors}\label{ssec:Tensor_Form}

A real tensor of order $d$ is defined as an element of finite dimensional Hilbert space
$\mathbb{W}_{\bf m} = \mathbb{R}^{I_1\times ... \times I_d}$
composed of the $d$-fold, $M_1\times ... \times M_d$ real-valued arrays (tensors), where
${\bf m}=(M_1,\ldots ,M_d)$, and $I_\ell:=\{1,...,M_\ell\}$, $\ell=1,...,d$.
A tensor ${\bf{A}}\in \mathbb{R}^{I_1\times ... \times I_d}$ 
is represented entry-wise by
\[
{\bf{A}}=[{a}(i_1,...,i_d)]\equiv [{a}(\mathbf{i})]\equiv [{a}_{i_1,...,i_d}]\equiv [{a}_\mathbf{i}]
\quad \mbox{with}\quad 
\mathbf{i}\in{\cal I}=I_1\times ... \times I_d.
%i_\ell\in I_\ell:=\{1,...,M_\ell\}, % \quad \mbox{and} \quad {\bf n}=(N_1,...,N_d).
\]
%where $\mathbf{i}\in{\cal I}=I_1\times ... \times I_d$.
The Euclidean scalar product, $\left\langle \cdot ,\cdot \right\rangle :
\mathbb{W}_{\bf m}\times \mathbb{W}_{\bf m}\to \mathbb{R}$,
is defined by
$$
\left\langle {\textbf{A}} , {\textbf{B}} \right\rangle:=
\sum_{{\bf i}\in {\cal I}} {a}_{\bf i} {b}_{\bf i}, \quad
{\bf{A}},{\bf{B}}\in \mathbb{W}_{\bf m}.
$$ 
The storage size for a $d$th order tensor scales exponentially in $d$, 
$\operatorname{dim}(\mathbb{W}_{{\bf m}})=M_1\cdots M_d$, that causes the so-called ``curse of dimensionality''.
In this section, for ease of presentation  we assume $M_\ell=M$ for $\ell=1,...,d$.

The efficient low-parametric representations of $d$th order tensors can be realized by using
low-rank separable decompositions (formats). 
The commonly used canonical and Tucker tensor formats \cite{KoldaB:07} 
are constructed by linear combination of
the simplest separable elements given by rank-$1$ tensors,
% \[
%  {\bf{A}}= \bigotimes_{\ell=1}^d {a}^{(\ell)}, \quad {a}^{(\ell)}\in \mathbb{R}^{M},
% \]
\[
{\bf U} = {\bf u}^{(1)}\otimes ... \otimes {\bf u}^{(d)}\in 
\mathbb{R}^{I_1 \times \ldots \times I_d},
\quad {\bf u}^{(\ell)}\in \mathbb{R}^{M_\ell},
\]
with entries $u_{i_1,\ldots i_d}= u^{(1)}_{i_1}\cdot \cdot\cdot u^{(d)}_{i_d}$, 
which can be stored by $d M$ numbers.

Tensor-structured numerical methods for PDEs were particularly initiated by employment of the 
canonical and Tucker tensor formats in grid based ``ab initio`` electronic structure calculations,
namely, for accurate evaluation of the 3D convolution integrals with the Newton kernel, 
see \cite{VeKhorTromsoe:15} and references therein. The literature overview on multi-linear algebra
and tensor numerical methods for PDEs  
can be found, for example, in \cite{KoldaB:07,khor-survey-2014,GraKresTo:13,dc-phd,VeKhorTromsoe:15}. 

In this paper we apply the factorized representation of $d$th order tensors in the 
tensor train (TT) format \cite{ot-tt-2009}, which is the particular case of 
the {\it matrix product states} (MPS) decomposition 
\cite{white-dmrg-1993,Vidal-Effic-Simul-Quant-comput-2003,Scholl:11}. 
The latter was introduced
long since in the physics community and successfully applied in quantum 
chemistry computations and in spin systems modeling.
For a given rank parameter ${\bf r}=(r_1,...,r_{d-1})$, and the respective
index sets $J_\ell=\{1,...,r_\ell\}$ ($\ell=1,...,d-1$),
%with the constraint $ J_0= J_d=\{1\}$ (i.e., $r_0=r_d=1$),
the rank-${\bf r}$ TT format contains all elements
${\bf{A}}=[{a}(i_1,...,i_d)]\in\mathbb{W}_{\bf m} $
%in $\mathbb{W}_{\bf n}=\mathbb{R}^{\cal I}$
which can be represented as the contracted products of $3$-tensors 
over the $d$-fold product index set ${\cal J}:=\times_{\ell=1}^{d-1} J_\ell$, such that
\[
{\bf{A}}=\sum\limits_{{(\alpha_1,...,\alpha_{d-1})}\in {\cal J}} 
{\bf a}^{(1)}_{1,\alpha_1}\otimes {\bf a}^{(2)}_{\alpha_1,\alpha_2} \otimes 
\cdots \otimes {\bf a}^{(d)}_{\alpha_{d-1},1},
\]
or entry-wise
\[
 a({\bf i})=\sum\limits_{(\alpha_1,...,\alpha_{d-1})={\bf 1}}^{\bf r}
 {a}^{(1)}_{1,\alpha_1}(i_1) {a}^{(2)}_{\alpha_1,\alpha_2}(i_2) 
\cdots  {a}^{(d)}_{\alpha_{d-1},1}(i_d)= A^{(1)}(i_1)A^{(2)}(i_2)\cdots A^{(d)}(i_d),
\]
with generating vectors  
${\bf a}^{(\ell)}_{\alpha_{\ell-1},\alpha_{\ell}}   \in \mathbb{R}^{M_\ell}$, 
and $r_{\ell-1}\times r_{\ell}$ matrices 
$A^{(\ell)}(i_\ell)=[{a}^{(\ell)}_{\alpha_{\ell-1},\alpha_{\ell}}(i_\ell)]$, ($\ell=1,...,d$)
under the convention $r_0=r_d=1$.
% Here the rank product operation ``$\Join$'' is defined as a regular 
% matrix product of the two core matrices, their fibers (blocks) being multiplied 
% by means of tensor product. % \cite{KazKhor_1:10}.
The TT representation reduces the storage cost to $\mathcal{O}(d r^2 M)$, $r=\max {r_\ell}$,
$M=\max {M_\ell}$.

It is often convenient to characterize the TT-rank ${\bf r}=(r_1,...,r_{d-1})$ with a single number. 
We therefore introduce the notion of the \emph{effective (average) rank} of a TT-tensor ${\bf{A}}$. 
In the case of equal mode sizes $M$ it is defined as the positive  solution of the quadratic equation
\begin{equation}\label{effrank}
r_1 + \sum_{k=2}^{d-1}r_{k-1}r_k + r_{d-1} = r + \sum_{k=2}^{d-1}r^2 + r
\end{equation}
and will be denoted by $r_{\text{eff}}$ or average QTT rank $r$.  %\vspace{\baselineskip}

\subsection{Quantized-TT approximation of function related vectors}\label{ssec:QTT_Approx}

In the case of large mode size $M$, the asymptotic storage for a $d$th order tensor 
can be reduced to logarithmic scale $\mathcal{O}(d \log M)$ by using quantics-TT (QTT) 
tensor approximation \cite{KhQuant:09}. 
In the present paper, we apply this 
approximation techniques to long $N_{ov}$-vectors representing the columns of $L_V$ factor and other
parts of the BSE matrix, as well as to eigenvectors of the BSE system.

The QTT-type approximation of an $M$-vector with $M=q^{d'}$, $d'\in \mathbb{N}$, $q=2,3,...$,
is defined as the tensor decomposition (approximation) in the TT or canonical format applied 
to a tensor obtained by the folding (reshaping) of the initial vector to an $d'$-dimensional 
$q\times \ldots \times q$ data array. The latter is thought as an element of the 
multi-dimensional quantized tensor 
space $\mathbb{Q}_{{q},d'}= \bigotimes_{j=1}^{d'}\mathbb{K}^{q},
 \; \mathbb{K}\in \{\mathbb{R},\mathbb{C}\}$, and $d'$ is the auxiliary dimension
 (virtual, in contrary to the real space dimension $d$)  parameter
that measures the depth of the quantization transform. A vector 
$
{\bf x}=[x(i)]_{i\in I}\in \mathbb{W}_{{M}},
$
is reshaped to its multi-dimensional quantics image in $\mathbb{Q}_{q,d'}$ by $q$-adic folding, 
\[
\mathcal{F}_{q,d'}:{\bf x} \to {\bf{Y}}
=[y({\bf j})]\in \mathbb{Q}_{q,d'}, \quad {\bf j}=\{j_{1},\ldots,j_{d'}\}, 
%\quad \mbox{with} \quad j_{\nu}\in \{1,2,\ldots,q\}, \,\nu=1,\ldots ,L, 
\]
with  $j_{\nu}\in \{1,\ldots, q\}$ for $\nu=1,...,L$.
Here for fixed $i$, we have $y({\bf j}):= x(i)$, and $j_\nu=j_\nu(i)$ is defined via $q$-coding,
$
j_\nu - 1= C_{-1+\nu}, 
$
such that the coefficients $C_{-1+\nu} $ are found from the
$q$-adic representation  of $i-1$ (binary coding for $q=2$),
\[
i-1 =  C_{0} +C_{1} q^{1} + \cdots + C_{d'-1} q^{d'-1}\equiv
\sum\limits_{\nu=1}^{d'} (j_{\nu}-1) q^{\nu-1}.
\]
Assuming that for the rank-${\bf r}$ TT approximation of the quantics image ${\bf{Y}}$
there  holds $r_k \leq r$, $k=1,\ldots ,L$, 
%\todo{You mean after approximating the quantics image in the TT format?} 
the complexity of this tensor representation ${\bf{Y}}$  reduces to the logarithmic scale
$$
q r^2 \log_q M \ll M. % \quad \mbox{where}\quad r_k \leq r, \quad k=1,\ldots ,L.
$$ 

The computational gain of the QTT approximation is justified by the 
perfect rank decomposition proven in \cite{KhQuant:09} for 
a wide class of function-related tensors obtained by sampling the corresponding functions
over a uniform or properly refined grid. This class of functions includes
complex exponentials, trigonometric functions, polynomials 
and Chebyshev polynomials, wavelet basis functions.
We refer to \cite{DKhOs-parabolic1-2012,osel-constr-2013,VeBoKh:Ewald:14,khor-survey-2014} 
for further results on QTT approximation and their application.

The QTT-type approximation to some $2^{d'} \times 2^{d'}$ matrices was introduced in
\cite{osel-2d2d-2010}. The construction and analysis of the QTT representation to 
the Laplacian related matrices is presented in \cite{khkaz-lap-2012}. 
The definition of Matrix Product Operator (MPO) is given in \S\ref{ssec:QTTsolver_TDA}.

In this paper we apply the QTT approximation method to the BSE eigenvalue problem,
where matrices and eigenvectors are transformed to the QTT representation and the arising 
high-dimensional eigenvalue problem is solved by using the block-TT tensor 
format \cite{DoKhSavOs_mEIG:13}.

 \subsection{Analysis of the QTT rank parameters for the BSE data}
\label{ssec:QTTranks_BSE}

The motivating point for the following considerations in this section was the curious 
numerical observation discussed in \cite{VeKhBoKhSchn:12,VeKhorMP2:13}.
It was demonstrated that the QTT ranks \cite{KhQuant:09} 
of column vectors in the Cholesky factor for the TEI tensor are almost equal to the fundamental 
structural characteristic of the molecular system, the number of occupied molecular 
orbitals $N_{o}$, i.e. do not depend
on the size $N_b^2$ of the TEI matrix, determined by the number of GTO basis functions $N_b$. 
This fact indicates the  existence of the tensor-structured QTT representation
for the Cholesky factors with the very mild complexity scaling in the matrix size $N_b^2$.

Here we demonstrate that the very similar property can be observed for the matrices and vectors 
involved in the BSE spectral problem.

First, we investigate numerically  QTT ranks of the long eigenvectors in BSE problem 
and the canonical QTT ranks in the skeleton vectors of the low-rank matrix factorizations in the case of 
compact molecules and chains of atoms.
Specifically, in numerical tests we found that
the QTT-ranks do not depend on the problem size $N_{ov}$ and, hence, on the number of GTO basis 
functions specifying the size of BSE system, but again depend only on the fundamental 
physical characteristics  of a molecular system, $N_{o}$.

Next Table \ref{tab:qtt_ranks} illustrates that for the TDA model applied to single molecules and 
to molecular chains
the average QTT ranks, computed for column vectors in $L_V$ factor in (\ref{eqn:L_V_factor}) 
and for $m_0=30$ senior TDA-eigenvectors, are almost equal or even smaller than 
the number of occupied molecular orbitals, $N_{o}$, in the system under consideration.
Notice that these results are obtained by compression of each column from $L_V$ or eigenvectors separately.
In the next section \S\ref{ssec:QTTsolver_TDA}, we apply the so-called block-TT format where
the meaning of QTT approximation is adapted to the subset of eigenvectors.

% \textcolor{red}{It's worth to note how exactly those results are obtained, in particular, 
% that here we compress each column from $L_v$ or eigenvectors separately.
% Because in DMRG computations, we use the Block-TT format, which is different.}
 
  \begin{table}[hbp]
 \begin{center}
 \begin{tabular}
[c]{|c|c|c|c|c|c|c|c|}%
\hline
Mol. sys. & H$_2$O& H$_{16}$ &N$_2$H$_4$ &C$_2$H$_5$OH &H$_{32}$  & C$_2$H$_5$ NO$_2$ &  C$_3$H$_7$ NO$_2$ \\
 \hline
$ N_{o}$     & $5$   & $8$    & $9$      & $13$ & $16 $  & $20$ & $24$   \\
\hline
QTT ranks of $L_V$& $5.4$   & $7$    & $9.1$    & $12.7$ & $14 $  & $17.5$  &  $21$   \\
\hline
QTT ranks of e-vectors & $5.3$ & $7.6$  & $9.1$ & $12.7$ & $13.6$  & $17.2$  &  $20.9$   \\
\hline
$N_{ov}$  & $180$  & $448$  & $657$   & $1430$ & $1792$  & $3000$  &  $4488$   \\
\hline
\end{tabular}
\end{center}
\caption{Average QTT ranks of column vectors in $L_V$ and $m_0$ senior eigenvectors in TDA problem.}
\label{tab:qtt_ranks}
\end{table}

Table \ref{tab:qtt_chain_vs_basis} demonstrates that the considerable variation of the basis size for 
fixed molecular systems of H$_{12}$ or H$_{24}$ chains (hence with fixed number $N_{o}$) 
practically does not change the QTT ranks of column vectors in $L_V$ factor in (\ref{eqn:L_V_factor}) 
(QTT ranks of BSE eigenvectors are almost the same, see Table \ref{tab:qtt_ranks}). 
\begin{center}%
\begin{table}[htb]
\begin{center}
\begin{tabular}
[c]{|c|c|r|r|r|r|}%
\hline
 H$_{12}$, $N_{o}=6$   &  $N_b$         & $36$    & $48$   & $72$   &  $84$ \\
 \cline{2-6}  &   size BSE     & $360^2$ & $504^2$ & $792^2$ &  $936^2$  \\
 \cline{2-6}   &  QTT ranks & $5.4$ & $6.5$ & $6.6$ & $7.0$    \\
\hline \hline
  H$_{24}$, $N_{o}=12$ &   $N_b$         & $72$    & $96$   & $144$   &  $168$ \\
 \cline{2-6} & size BSE       & $1440^2$ & $2016^2$ & $3168^2$ & $3744^2$   \\
 \cline{2-6} &  QTT ranks & $9.5$ & $11.6$ & $11.8$ & $12.7$   \\
 \hline 
% \hline\hline
\end{tabular}
\end{center}
\caption{Average QTT ranks of column vectors in $L_V$ factor vs. $N_{o}$ and the BSE-size for 
Hydrogen chains: illustrates weak dependence on the number of basis functions $N_b$.}
\label{tab:qtt_chain_vs_basis}
\end{table}
\end{center}

Figure \ref{fig:QTTranks_vs_Norb} indicates that the behavior of QTT ranks in the column vectors of $L_V$-factor
reproduces the system size $N_{ov}$  in terms of $N_{o}$ on the logarithmic scale.

\begin{figure}[htbp]
\centering
\includegraphics[width=7.1cm]{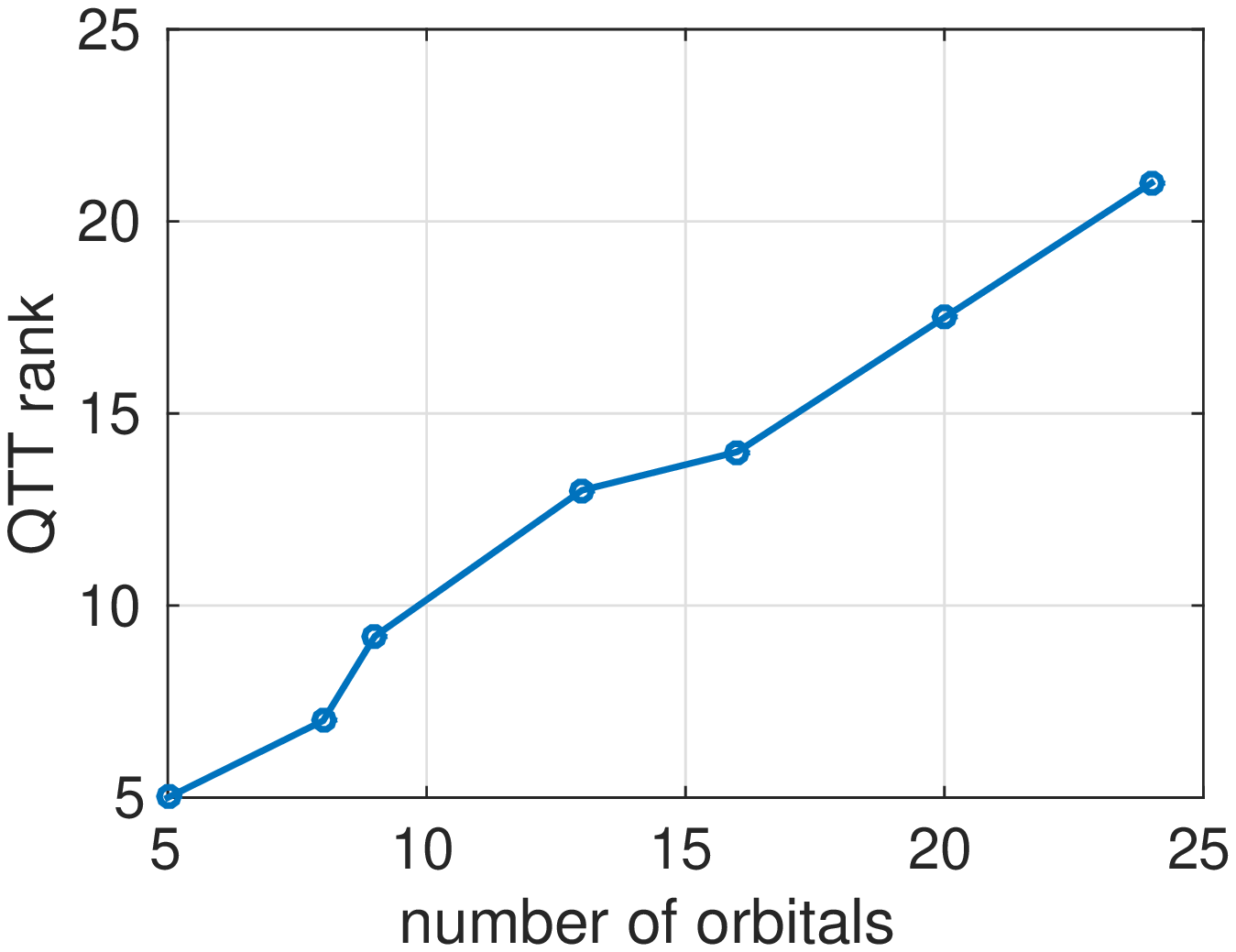} \quad
\includegraphics[width=7.1cm]{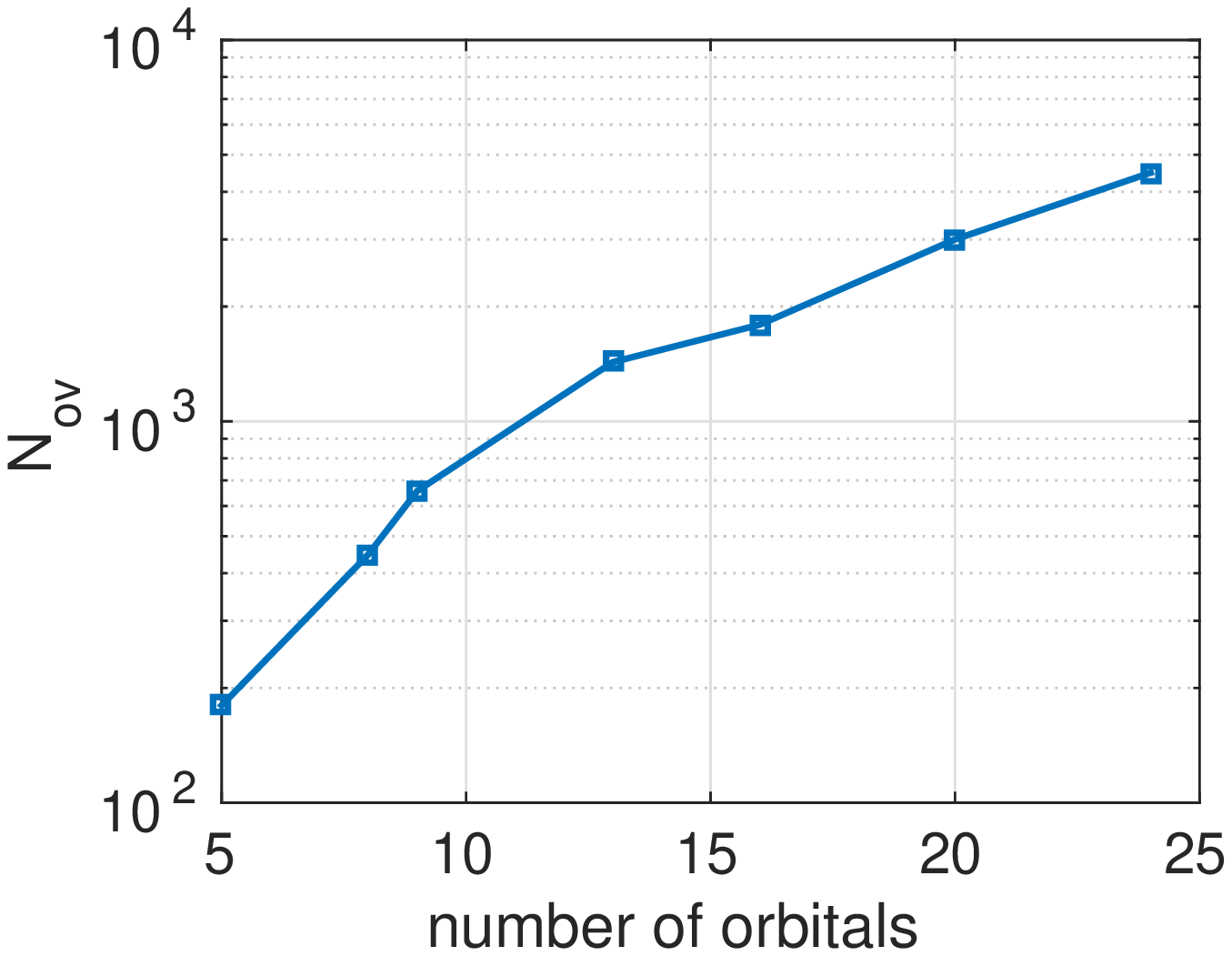}
\caption{\small QTT ranks (left) and  $N_{ov}$ on logarithmic scale (right) 
vs. the number of orbitals, $N_{o}$.}
\label{fig:QTTranks_vs_Norb}
\end{figure}

It is worth to note that in the case of single molecules the commonly used number of GTO basis sets satisfy the 
relation $N_{b}/N_{o}\geq C_{GTO} \approx 10$ (see examples below), which implies the asymptotic 
behavior $N_{ov}\approx C_{GTO} N_{o}^2$. Hence, the QTT rank estimate $r_{QTT}\approx N_{o}$ 
obtained above leads to the asymptotic complexity of the QTT-based tensor solver,
\begin{equation}\label{eqn:BSE_QTTcost}
{\cal W}_{BSE}=\mathcal{O}(\log(N_{ov})  r_{QTT}^2)=\mathcal{O}(\log(N_{o})N_{o}^2),
\end{equation} 
 which is asymptotically on the same scale as that for the data-structured algorithms 
 based on full-vector arithmetics (see Section \ref{sec:Iter_Struct_BSE}).
The same observation applies to the chain type molecular systems.

However, the high precision Hartree-Fock calculations may require much larger GTO basis sets 
so that the constant $C_{GTO}$ may increase considerably. 
In this situation the QTT-based tensor approach seems to outperform 
the algorithms in full-vector arithmetics.  

Even more important consequence of (\ref{eqn:BSE_QTTcost}) is that 
the rank behavior $r_{QTT}\approx N_{o}$ 
indicates that the QTT tensor-based algorithm has the memory requirements and 
the algebraic complexity $\mathcal{O}(\log (N_{o}) N_{o}^2)$ 
depending only on the fundamental physical characteristics of the molecular system, 
the number of occupied molecular orbitals, $N_{o}$ (but not on the system size $N_{ov}^2$). 
Hence, we introduce the {\it hypothesis}: estimate (\ref{eqn:BSE_QTTcost}) determines the 
irreducible lower bound on the asymptotic algebraic complexity of the large scale BSE eigenvalue problem.

\subsection{Block-TT eigenvalue solver in high-dimensional QTT format}
\label{ssec:QTTsolver_TDA}

Since the eigenvectors of the TDA problem exhibit moderate QTT ranks, it is tempting to 
apply the TT eigenvalue solver, such as the DMRG algorithm \cite{white-dmrg-1993,Scholl:11}.
As we are always looking for several eigenvectors, we can use the accelerated version 
\cite{DoKhSavOs_mEIG:13}, where only one TT block is considered at once.
The rank adaptivity (as well as fast convergence) comes from the separation of the eigenvalue 
enumerator from the original index, when we replace the enumerator to the next TT block.

This algorithm can be adapted to the input data we have in the Bethe-Salpeter problem.
In the general setting, given an eigenvalue problem $A U = U \Lambda$,
the method assumes that the matrix is given in the matrix TT (also called as Matrix Product Operator) format
\begin{equation}
A(\mathbf{i},\mathbf{j}) =  \mathbf{A}^{(1)}(i_1,j_1)\mathbf{A}^{(2)}(i_2,j_2) 
\cdots \mathbf{A}^{(d')}(i_{d'},j_{d'}),
\label{eq:ttm}
\end{equation}
where $\mathbf{i}$ and $\mathbf{j}$ are multi-indices comprised from 
$i_1,\ldots,i_{d'}$ and $j_1,\ldots,j_{d'}$, respectively.
Each term $\mathbf{A}^{(l)}(i_l,j_l)$ in the right-hand side is a $r_{l-1} \times r_l$ matrix, 
similarly to the ``vector'' TT format, but parametrized by two original indices $i_l,j_l$, 
$1\leq i_l,j_l \leq q_\ell$.
Here we use the general notation $d'$ for the dimension parameter used in the description 
of QTT format in \S\ref{ssec:QTT_Approx}. 
The mode size $M_\ell$ in the general definition of TT format is substituted by $q_\ell$ 
for the QTT tensors.

A slight generalization of the QTT format introduced in Section \ref{ssec:QTT_Approx} involves 
different \emph{prime} dimensions of a tensor, instead of the same value $q$.
Given initial dimensions $N_o$ and $N_v$, we decompose these numbers into smallest nontrivial prime factors, say,
$$
N_o = q_1\cdots q_o, \quad N_v = q_{o+1} \cdots q_{d'},
$$
such that the total problem size $N_{ov} = q_1\cdots q_{d'}$ yields the corresponding index factorization, 
allowing the TT format \eqref{eq:ttm}.
If $N_o$ and $N_v$ are powers of $2$, we end up with the classical QTT format with 
$2 \times \cdots \times 2$-tensors.
But in a more general case, any other small factors (like $3,5,7$, and so on) are possible.

The eigenvectors are sought in the block QTT format
\begin{equation}
U_{m}(\mathbf{i}) = \mathbf{U}^{(1)}(i_1)\cdots \mathbf{U}^{(\ell-1)}(i_{\ell-1}) 
\mathbf{\hat U}^{(\ell)}(i_{\ell}, m)  \mathbf{U}^{(\ell+1)}(i_{\ell+1}) \cdots \mathbf{U}^{(d')}(i_{d'}),
\label{eq:btt}
\end{equation}
where $\mathbf{\hat U}^{(\ell)}$ is a special TT block, containing the eigenvector enumerator $m=1,\ldots,m_0$.
Using the SVD, one can decompose $\mathbf{\hat U}^{(\ell)}$ and move $m$ to a neighboring 
block \cite{DoKhSavOs_mEIG:13}.
In this paper, we only need to know that the DMRG technique is a Galerkin projection method:
 the remaining blocks $\mathbf{U}^{(p)}$, $p\neq \ell$, constitute the \emph{frame} 
 matrix $U_{\neq \ell}\in\mathbb{R}^{N_{ov} \times r_{\ell-1}q_\ell r_\ell},$
$$
U_{\neq \ell}(\mathbf{i}, \alpha_{\ell-1}j_{\ell}\alpha_{\ell}) = \mathbf{U}^{(1)}(i_1)\cdots 
\mathbf{U}^{(\ell-1)}_{:,\alpha_{\ell-1}}(i_{\ell-1}) 
\delta_{i_\ell,j_\ell}  \mathbf{U}^{(\ell+1)}_{\alpha_\ell,:}(i_{\ell+1}) \cdots \mathbf{U}^{(d')}(i_{d'}),
$$
% \textcolor{red}{It seems that here we need in addition
% (a) comments on the value of $N$ above, 
% (b) specify the size of the local matrix $U_{\neq \ell}^\top A U_{\neq \ell}$ and $\hat u^{(\ell)}$  
% (maybe introducing notation for that?), 
% (c) comment on definition of $U_{\neq \ell}$ in terms of previous iterand, and
% (d) specify the dependance on parameter $m_0$} 
% \textcolor{green!50!black}{
% \begin{itemize}
% \item (a) Sorry, $N$ is in fact $N_{ov}$. Corrected.
% \item (b) The size of $u^{(\ell)}$ is written, all other sizes now follow by consistency.
% \item (c) done
% \item (d) it's in $u^{(\ell)}$, done.
% \end{itemize}
% }
such that the \emph{local} problem reads
\begin{equation}
\left(U_{\neq \ell}^\top A U_{\neq \ell}\right)  \hat u^{(\ell)} = \hat u^{(\ell)} \Lambda, 
\qquad \hat u^{(\ell)} \in \mathbb{R}^{r_{\ell-1}q_\ell r_\ell \times m_0},
\label{eq:locsys}
\end{equation}
where the diagonal $\Lambda$ contains the Ritz values, approximating the eigenvalues of the original problem.
After solving this problem, the block $\mathbf{\hat U}^{(\ell)}$ is populated with the elements 
of $\hat u^{(\ell)}$, and the process continues for the next block.

\begin{center}%
\begin{table}[htb]
\begin{center}
\begin{tabular}
[c]{|c|c|r|r|r|r|}%
\hline
 H$_{12}$, $N_{o}=6$,   & $N_b$                                   & $36$   & $48$ & $72$   &  $84$ \\
 \cline{2-6} 1 DMRG iter  &  CPU time                              & 0.019   & 0.02   & 0.034   & 0.04    \\
 \cline{2-6}  &  av. QTT rank                                      & 19.0   & 20.2 & 22.0   & 22.6    \\
 \cline{2-6}  &  $\frac{\mathtt{mem}(\mbox{QTT})}{N_{ov}m_0}$ & 1.07 & 1.00   & 0.94   & 0.92    \\
 \cline{2-6}  &  $\frac{\|\mu_{qtt}-\mu_\star\|}{\|\mu_\star\|}$ & 2.86e-2& 1.22e-2& 4.60e-3& 8.41e-3 \\
\hline \hline
 H$_{12}$, $N_{o}=6$,  &  CPU time                                 & 0.02   & 0.04   & 0.06   & 0.08    \\
 \cline{2-6} 2 DMRG iters  &  av. QTT rank                         & 9.7    & 14.5   & 14.7 & 13.9    \\
 \cline{2-6}  &  $\frac{\mathtt{mem}(\mbox{QTT})}{N_{ov}m_0}$ & 0.25 & 0.35   & 0.23   & 0.18    \\
 \cline{2-6}  &  $\frac{\|\mu_{qtt}-\mu_\star\|}{\|\mu_\star\|}$ & 3.29e-3& 6.36e-3& 5.84e-3& 7.03e-3 \\
\hline \hline
  H$_{24}$, $N_{o}=12$ & $N_b$ & $72$    & $96$ & $144$   &  $168$ \\
 \cline{2-6} 1 DMRG iter  &  CPU time                              & 0.10    & 0.17   & 0.09 & 0.12 \\
 \cline{2-6}  &  av. QTT rank                                      & 21.8    & 22.5 & 23.5 & 23.7 \\
 \cline{2-6}  &  $\frac{\mathtt{mem}(\mbox{QTT})}{N_{ov}m_0}$ & 0.42 & 0.36   & 0.66 & 0.74 \\
 \cline{2-6}  &  $\frac{\|\mu_{qtt}-\mu_\star\|}{\|\mu_\star\|}$ & 1.95e-1 & 1.10e-1& 6.8e-2 & 5.8e-2 \\
\hline \hline
 H$_{24}$, $N_{o}=12$  &  CPU time                                 & 0.06    & 0.1 & 0.23 & 0.21 \\
 \cline{2-6} 2 DMRG iters  &  av. QTT rank                         & 13.5    & 19.8   & 17.7& 17.8 \\
 \cline{2-6}  &  $\frac{\mathtt{mem}(\mbox{QTT})}{N_{ov}m_0}$ & 0.14 & 0.20       & 0.3  & 0.3 \\
 \cline{2-6}  &  $\frac{\|\mu_{qtt}-\mu_\star\|}{\|\mu_\star\|}$ & 6.43e-3 & 9.50e-3  & 8.69e-3 &  8.97e-3 \\
 %\cline{2-6}  &  $\|\mu_{qtt}-\mu_\star\| $  in eV               & 0.34   & 0.03  &  0.23 & 0.14  \\
\hline
\end{tabular}
\end{center}
\caption{DMRG iteration in block-QTT format for TDA model with $m_0=30$ sought eigenvalues and all 
low-rank approximation thresholds $0.1$. $\mu_\star$ is computed for the exact 
TDA matrix {\eqref{eqn:BSE-Tamm-Danc}}.}
\label{tab:dmrg}
\end{table}
\end{center}

Numerical experiments show that $\boldsymbol{\Delta\varepsilon}$ and $\overline{W}$ 
are well compressible in the matrix QTT format \eqref{eq:ttm}.
However, this is not the case for $V$.
We utilize the fact that $V$ can be well approximated by a low-rank matrix, $V = L_V L_V^\top$.
The factor $L_V$ has the same conceptual meaning as the eigenvectors: it is a horizontal stack of 
$R$ vectors of length $N$.
Hence we can use the block TT format \eqref{eq:btt} for $L_V$ (replacing $U$ by $L_V$).
It is even easier since we do not need to move the enumerator $m$, but can fix it in the last TT block.
In each DMRG step, the projected matrix \eqref{eq:locsys} is constructed as 
$$
U_{\neq \ell}^\top A U_{\neq \ell} = U_{\neq \ell}^\top \boldsymbol{\Delta\varepsilon} U_{\neq \ell}
+ \left(U_{\neq \ell}^\top L_V\right) \left(U_{\neq \ell}^\top L_V\right)^\top - 
U_{\neq \ell}^\top \overline{W} U_{\neq \ell},
$$
where each product is implemented in a fast way, using the TT formats of $U_{\neq \ell}$, 
$\boldsymbol{\Delta\varepsilon}$, $L_V$ and $\overline{W}$.
\begin{remark}\label{rem:W_2_QTT_direct}
Note that here $\overline{W}$ is the original matrix from \eqref{eq:AB_ex}, compressed 
in the matrix QTT format \eqref{eq:ttm}.
No additional low-rank or block-diagonal constraints are imposed.
Therefore the results of the DMRG method in this section should be compared 
directly to the result of the exact eigenvalue solver.
\end{remark}

The reduced eigenvalue problem \eqref{eq:locsys} has the size $r_{\ell-1}q_\ell r_\ell$ and 
can be solved using the full \texttt{eig}.
The only explicitly iterative part is a sweep over different TT blocks in the alternating fashion.
By ``iteration'', we mean the sequential sweep from the first to the $d'$-th TT block, 
or the other way around.

The numerical results are presented in Table \ref{tab:dmrg}: 
CPU time (sec.), average QTT rank, memory ratio 
(the storage of the QTT format over the total number of elements in the full representation) 
and the relative error of the eigenvalues.
We use the tolerance $10^{-6}$ to compress $\boldsymbol{\Delta\varepsilon}$ into 
the matrix TT format\footnote{This accuracy is necessary, since $\boldsymbol{\Delta\varepsilon}$ 
is the dominant part of the matrix. Fortunately, the TT ranks 
of $\boldsymbol{\Delta\varepsilon}$ are below $10$ even for such accuracy, 
whereas the ranks of $L_V$ and $\overline{W}$ may exceed a hundred.}, but for all other approximations, 
including the factorization $V=L_VL_V^\top$, the tolerance is set to $\varepsilon = 0.1$.
We notice that one DMRG iteration gives insufficient accuracy of the solution, 
but the second iteration delivers a relative error below the theoretical estimate $\varepsilon^2$.
The CPU time is comparable or smaller than the time of the best Sherman-Morrison 
inversion methods in the previous section, as demonstrated in Table \ref{tab:times_QTT_TDA} 
(cf. Table \ref{tab:times_TDABSE_Blsp}). Recall that the row ``absolute error'' 
in Table \ref{tab:times_QTT_TDA} represents the quantity 
$\|\mu_{qtt}-\mu_\star\|= ({\sum}_{m=1}^{m_0} (\mu_{qtt,m}-\mu_{\star,m})^2)^{1/2} $
characterizing the total absolute error in the first $m_0$ eigenvalues 
calculated in the Euclidean norm.  
\begin{table}[htbp]
 \begin{center}
 \begin{tabular}
[c]{|c|c|c|c|c|c| }%
\hline
Molecular syst.&  C$_2$H$_5$OH & H$_{32}$  & C$_2$H$_5$ NO$_2$ & H$_{48}$ & C$_3$H$_7$ NO$_2$ \\
         \hline\hline
 TDA  size & $1430^2$   &  $1792^2$    & $3000^2$  & $4032^2$ & $4488^2$   \\
 \hline
time QTT eig &  $0.14$   & $0.23$      & $0.32$ & $0.28$    &  $0.63$   \\
\hline
abs. error (eV)  &  $0.08$ & $0.19$      & $0.17$ & $ 0.14$ &  $ 0.00034$   \\
\hline
 \end{tabular}
\end{center}
\caption{Time (s) and absolute error (eV) for QTT-DMRG eigensolvers for TDA matrix. }
\label{tab:times_QTT_TDA}
\end{table}

The QTT format provides also a considerable reduction of memory needed to store eigenvectors.

\section{Conclusions}\label{sec:Conclusions}

This paper presents efficient iterative solution of the Bethe-Salpeter large-scale eigenvalue
problem using the reduced basis approach via low-rank factorizations introduced in \cite{BeKhKh_BSE:15}.

For the static screen interaction part of the BSE submatrix, which was problematic
for the low-rank representation in \cite{BeKhKh_BSE:15}, we have found a beneficial 
substitution by a small sub-block, which reduces the approximation error by the order of magnitude.
Moreover, it provides the two-sided error bounds for the exact BSE excitation energies
in the case of compact and chain-type molecular systems.

We show that the structured inverse iterations (by using matrix inverse) provide fast convergence
for calculation of the required central part of the BSE spectrum.
For both  BSE and TDA models,  the inverse matrix can be represented in the same 
diagonal plus low-rank plus reduced-block format by using the Sherman-Morrison scheme.
The estimates on the complexity of algorithms for diagonal plus low-rank plus
reduced-block inverse iterations are presented in Lemmas 4.1 and 4.3.

Solution of the BSE spectral problem in the QTT format is disscussed in detail.
The QTT tensor transform of the initial BSE system to the higher dimensional
setting allows to construct a structural solver of the complexity $\mathcal{O}(\log(N_{o})N_{o}^2)$, 
 see (\ref{eqn:BSE_QTTcost}). This complexity it is determined by only
the number of occupied orbitals, $N_{o}$, in the molecular system 
(i.e. on physical characteristics of the molecule), but it is almost independent of the
system size determined by the number of atomic orbitals basis functions, $N_b$.
%The results are confirmed  by a number of numerical tests conducted through out the paper
%for various moderate size molecules and for molecular chains. 
In numerical tests we observe dramatical reduction of solution time.
For example, TDA calculations in QTT format for C$_2$H$_5$OH molecule
with matrix size $1430^2$ take $0.14$ sec,
while for C$_3$H$_7$NO$_2$ (Alanine amino acid) with TDA matrix size $4488^2$
CPU time increases only to $0.63$ sec.

The results are confirmed  by a number of numerical tests conducted through out the paper
for various moderate size molecules and molecular chains. 
Note that the solution of the eigenvalue problem with the
rank-structured representation of the BSE matrix reduces calculation times
for large enough molecules at least by two orders of magnitude, 
see, for example, Table \ref{tab:times_BSE}, where for Alanine amino-acid,
with the matrix size $8976^2$, direct calculation takes $903$ s, 
while the low-rank iteration takes $4$ s.
% The numerical precision in this case is about $0.1$ eV.
Further reduction of complexity is achieved when using the DMRG-type iteration 
in the block-QTT tensor format, see Tables \ref{tab:dmrg}, \ref{tab:times_QTT_TDA}.

Several directions for future research work on the rank-structured reduced basis method for 
computation of excitation energies of molecules and solids will be considered.
Particularly, this includes 
comprehending the considered BSE model by some additional correction terms,
developments of the new data-sparse matrix structures,
and further applications of algorithms to large and lattice-structured molecular systems.

\begin{footnotesize}

\bibliographystyle{abbrv}
\bibliography{BSE_Fock}

\begin{thebibliography}{10}

\bibitem{BaiLi:12}
Z.~Bai and R.-C. Li.
\newblock Minimization principle for linear response eigenvalue problem, {I}:
  Theory.
\newblock {\em SIAM J. Matrix Anal. Appl.}, 33(4):10751100, 2012.

\bibitem{BaiLi:13}
Z.~Bai and R.-C. Li.
\newblock Minimization principle for linear response eigenvalue problem, {II}:
  Computation.
\newblock {\em SIAM J. Matrix Anal. Appl.}, 34(2):392--416, 2013.

\bibitem{BeFa:97}
P.~Benner and H.~F{a\ss b}ender.
\newblock An implicitly restarted symplectic {L}anczos method for the
  {H}amiltonian eigenvalue problem.
\newblock {\em Linear Algebra Appl.}, 263:75--111, 1997.

\bibitem{BeFaYa:15}
P.~Benner, H.~Fa\ss{}bender, and C.~Yang.
\newblock Some remarks on the complex {$J$}-symmetric eigenproblem.
\newblock Preprint MPIMD/15-12, Max Planck Institute Magdeburg, July 2015.

\bibitem{BeKhKh_BSE:15}
P.~Benner, V.~Khoromskaia, and B.~N. Khoromskij.
\newblock A reduced basis approach for calculation of the {B}ethe-{S}alpeter
  excitation energies using low-rank tensor factorizations.
\newblock {\em Molecular Physics, DOI: 10.1080/00268976.2016.1149241},
  (arXiv:1505.02696v1, 2015), 2016.

\bibitem{BeMeXu:98}
P.~Benner, V.~Mehrmann, and H.~Xu.
\newblock A numerically stable, structure preserving method for computing the
  eigenvalues of real {H}amiltonian or symplectic pencils.
\newblock {\em Numerische Mathematik}, 78(3):329--358, 1998.

\bibitem{BunByeMehrm:92}
A.~Bunse-Gerstner, R.~Byers, and V.~Mehrmann.
\newblock A chart of numerical methods for structured eigenvalue problems.
\newblock {\em SIAM J. Matrix Anal. Appl.}, (13):419--453, 1992.

\bibitem{BGFa:15}
A.~Bunse-Gerstner and H.~Fa{\ss}bender.
\newblock Breaking {Van L}oan's curse: A quest for structure-preserving
  algorithms for dense structured eigenvalue problems.
\newblock In P.~Benner, M.~Bollh\"ofer, D.~Kressner, C.~Mehl, and T.~Stykel,
  editors, {\em Numerical Algebra, Matrix Theory, Differential-Algebraic
  Equations and Control Theory}, pages 3--23. Springer International
  Publishing, 2015.

\bibitem{DeSaStJaCoLo:12}
J.~Deslippe, G.~Samsonidze, D.~A. Strubbe, M.~Jain, M.~L. Cohen, and S.~Louie.
\newblock Berkeley{GW}: A massively parallel computer package for the
  calculation of the quasi-particle and optical properties of materials and
  nanostructures.
\newblock {\em Comp. Phys. Communications}, 183:1269--1289, 2012.

\bibitem{DoKhSavOs_mEIG:13}
S.~Dolgov, B.~Khoromskij, D.~Savostyanov, and I.~Oseledets.
\newblock Computation of extreme eigenvalues in higher dimensions using block
  tensor train formats.
\newblock {\em Comp. Phys. Communications}, 185(4):1207--1216, 2014.

\bibitem{dc-phd}
S.~V. Dolgov.
\newblock {\em Tensor product methods in numerical simulation of
  high-dimensional dynamical problems}.
\newblock PhD thesis, University of Leipzig, 2014.

\bibitem{DKhOs-parabolic1-2012}
S.~V. Dolgov, B.~N. Khoromskij, and I.~V. Oseledets.
\newblock Fast solution of multi-dimensional parabolic problems in the tensor
  train/quantized tensor train--format with initial application to the
  {Fokker}-{Planck} equation.
\newblock {\em SIAM J. Sci. Comput.}, 34(6):A3016--A3038, 2012.

\bibitem{FaKre:06}
H.~Fa\ss{}bender and D.~Kressner.
\newblock Structured eigenvalue problem.
\newblock {\em GAMM Mitteilungen}, 29(2):297--318, 2006.

\bibitem{GraKresTo:13}
L.~Grasedyck, D.~Kressner, and C.~Tobler.
\newblock A literature survey of low-rank tensor approximation techniques.
\newblock {\em arXiv:1302.7121v1}, 2013.

\bibitem{Hedin}
L.~Hedin.
\newblock New method for calculating the one-particle {Green's} function with
  application to the electron-gas problem.
\newblock {\em Phys. Rev.}, 139:A796, 1965.

\bibitem{khkaz-lap-2012}
V.~A. Kazeev and B.~N. Khoromskij.
\newblock Low-rank explicit {QTT} representation of the {L}aplace operator and
  its inverse.
\newblock {\em SIAM J. Matrix Anal. Appl.}, 33(3):742--758, 2012.

\bibitem{vekh:13}
V.~Khoromskaia.
\newblock Black box {H}artree-{F}ock solver by the tensor numerical methods.
\newblock {\em Comp. Methods in Applied Math.}, 14:89--111, 2014.

\bibitem{VeBoKh:Ewald:14}
V.~Khoromskaia and B.~N. Khoromskij.
\newblock Grid-based lattice summation of electrostatic potentials by assembled
  rank-structured tensor approximation.
\newblock {\em Comp. Phys. Communications}, 185(12):3162--3174, 2014.

\bibitem{VeKhorMP2:13}
V.~Khoromskaia and B.~N. Khoromskij.
\newblock {M}{\o}ller-{P}lesset ({M}{P}2) energy correction using tensor
  factorizations of the grid-based two-electron integrals.
\newblock {\em Comp. Phys. Communications}, 185(1):2--10, 2014.

\bibitem{VeKhorTromsoe:15}
V.~Khoromskaia and B.~N. Khoromskij.
\newblock Tensor numerical methods in quantum chemistry: from {Hartree-Fock} to
  excitation energies.
\newblock {\em Phys. Chem. Chem. Phys.}, 17:31491 -- 31509, 2015.

\bibitem{VeKhBoKhSchn:12}
V.~Khoromskaia, B.~N. Khoromskij, and R.~Schneider.
\newblock Tensor-structured calculation of two-electron integrals in a general
  basis.
\newblock {\em SIAM J. Sci. Comput.}, 35(2):A987--A1010, 2013.

\bibitem{KhQuant:09}
B.~N. Khoromskij.
\newblock ${O}(d\log {N})$-quantics approximation of ${N}$-$d$ tensors in
  high-dimensional numerical modeling.
\newblock {\em J. Constr. Approx.}, 34(2):257--289, 2011.

\bibitem{khor-survey-2014}
B.~N. Khoromskij.
\newblock Tensor {N}umerical {M}ethods for {M}ultidimensional {PDE}s: {B}asic
  {T}heory and {I}nitial {A}pplications.
\newblock {\em ESAIM: Proceedings and Surveys, N. Champagnat, T. Leli{\'e}vre,
  A. Nouy, eds}, 48:1--28, January 2015.

\bibitem{knyazev-lobpcg-2001}
A.~V. Knyazev.
\newblock Toward the optimal preconditioned eigensolver: Locally optimal block
  preconditioned conjugate gradient method.
\newblock {\em SIAM Journal on Scientific Computing}, 23(2):517--541, 2001.

\bibitem{KoldaB:07}
T.~Kolda and B.~W. Bader.
\newblock Tensor decompositions and applications.
\newblock {\em SIAM Review}, 51(3):455--500, 2009.

\bibitem{SBotti:14}
S.~K\"orbel, P.~Boulanger, I.~Duchemin, X.~Blase, M.~AL~Marques, and S.~Botti.
\newblock Benchmark many-body {GW} and {B}ethe-{S}alpeter calculations for
  small transition metal molecules.
\newblock {\em Journal of Chemical Theory and Computation}, 10(9):3934--3943,
  2014.

\bibitem{Kressner_book:05}
D.~Kressner.
\newblock {\em Numerical Methods for General and Structured Eigenvalue
  Problems}, volume~46 of {\em Lecture Notes in Computational Science and
  Engineering}.
\newblock Springer, Berlin/Heidelberg, 2005.

\bibitem{LinSaadYa:15}
L.~Lin, Y.~Saad, and C.~Yang.
\newblock Approximating spectral densities of large matrices.
\newblock {\em ArXiv:1308.5467.v2.}, 2015.

\bibitem{MMMM}
D.~Mackey, N.~Mackey, C.~Mehl, and V.~Mehrmann.
\newblock Structured polynomial eigenvalue problems: good vibrations from good
  linearizations.
\newblock {\em SIAM J. Matrix Anal. Appl.}, 28(4):1029--1051, 2006.

\bibitem{MMT03}
D.~Mackey, N.~Mackey, and F.~Tisseur.
\newblock Structured tools for structured matrices.
\newblock {\em Electronic Journal of Linear Algebra (ELA)}, 10:106--145, 2003.

\bibitem{Meh08}
C.~Mehl.
\newblock On asymptotic convergence of nonsymmetric {Ja}cobi algorithms.
\newblock {\em SIAM J. Matrix Anal. Appl.}, 30:291--311, 2008.

\bibitem{NaPoSaad:13}
E.~Napoli, E.~Polizzi, and Y.~Y.~Saad.
\newblock Efficient estimation of eigenvalue counts in an interval.
\newblock {\em arXiv:1308.4275v2}, 2014.

\bibitem{OniReRu:02}
G.~Onida, L.~Reining, and A.~Rubio.
\newblock Electronic excitations: density-functional versus many-body
  {G}reen's-function approaches.
\newblock {\em Rev. of Modern Physics}, 74:601--659, 2002.

\bibitem{osel-2d2d-2010}
I.~V. Oseledets.
\newblock Approximation of $2^d \times2^d$ matrices using tensor decomposition.
\newblock {\em SIAM J. Matrix Anal. Appl.}, 31(4):2130--2145, 2010.

\bibitem{osel-constr-2013}
I.~V. Oseledets.
\newblock Constructive representation of functions in low-rank tensor formats.
\newblock {\em Constr. Appr.}, 37(1):1--18, 2013.

\bibitem{ot-tt-2009}
I.~V. Oseledets and E.~E. Tyrtyshnikov.
\newblock Breaking the curse of dimensionality, or how to use {S}{V}{D} in many
  dimensions.
\newblock {\em SIAM J. Sci. Comput.}, 31(5):3744--3759, 2009.

\bibitem{ReTouSa1:13}
E.~Rebolini, J.~Toulouse, and A.~Savin.
\newblock Electronic excitation energies of molecular systems from the
  {B}ethe-{S}alpeter equation: Example of {H}$_2$ molecule.
\newblock {\em In: Concepts and Methods in Modern Theoretical Chemistry (S.
  Ghosh and P. Chattaraj eds), vol 1: Electronic Structure and Reactivity},
  page 367, 2013.

\bibitem{ReToTeHeSa:15}
E.~Rebolini, J.~Toulouse, A.~M. Teale, T.~Helgaker, and A.~Savin.
\newblock Calculating excitation energies by extrapolation along adiabatic
  connections.
\newblock {\em Phys. Rev. A}, 91:032519, 2015.

\bibitem{ReOlRuOni:02}
L.~Reining, V.~Olevano, A.~Rubio, and G.~Onida.
\newblock Excitonic effects in solids described by time-dependent density
  functional theory.
\newblock {\em Phys. Rev. Lett.}, 88:066404, 2002.

\bibitem{RoGeSaBa:08}
D.~Rocca, R.~Gebauer, Y.~Saad, and S.~Baroni.
\newblock Turbo charging time-dependent density-functional theory with
  {L}anczos chains.
\newblock {\em J. Chem. Phys.}, 128:154104, 2008.

\bibitem{RoLuGa:10}
D.~Rocca, D.~Lu, and G.~Galli.
\newblock \emph{Ab Initio} calculations of optical absorption spectra:
  {S}olution of the {B}ethe-{S}alpeter equation within density matrix
  perturbation theory.
\newblock {\em J. Chem. Phys.}, 133:164109 1--10, 2010.

\bibitem{BeSa:51}
E.~E. Salpeter and H.~A. Bethe.
\newblock A relativistic equation for bound-state problems.
\newblock {\em Phys. Review}, 82(2):309--310, 1951.

\bibitem{SchGluHaBe:03}
W.~G. Schmidt, S.~Glutsch, P.~H. Hahn, and F.~Bechstedt.
\newblock Efficient ${O}({N}^2)$ method to solve the {B}ethe-{S}alpeter
  equation.
\newblock {\em Phys. Review B}, 67:085307, 2003.

\bibitem{Scholl:11}
U.~Schollw\"ock.
\newblock The density-matrix renormalization group in the age of matrix product
  states.
\newblock {\em Ann.Phys.}, 51(326):96--192, 2011.

\bibitem{ShJoYaDeLo:16}
M.~Shao, F.~H.~da Jornada, C.~Yang, J.~Deslippe, and S.~Louie.
\newblock Structure preserving parallel algorithms for solving the
  {B}ethe-{S}alpeter eigenvalue problem.
\newblock {\em Linear Algebra and its Applications}, 488:148--167, 2016.

\bibitem{Vidal-Effic-Simul-Quant-comput-2003}
G.~Vidal.
\newblock {Efficient classical simulation of slightly entangled quantum
  computations}.
\newblock {\em Phys. Rev. Lett.}, 91(14), 2003.

\bibitem{white-dmrg-1993}
S.~R. White.
\newblock Density-matrix algorithms for quantum renormalization groups.
\newblock {\em Phys. Rev. B}, 48(14):10345--10356, 1993.

\end{thebibliography}
\end{footnotesize}
\end{document}